\documentclass[12pt]{amsart}       
\usepackage{txfonts}
\usepackage{amssymb}
\usepackage{eucal}
\usepackage{bbm}
\usepackage{graphicx}
\usepackage{amsmath}
\usepackage{amscd}
\usepackage[all]{xy}           
\usepackage{amsfonts,latexsym}
\usepackage{xspace}
\usepackage{epsfig}
\usepackage{float}
\usepackage{color}
\usepackage{fancybox}
\usepackage{colordvi}
\usepackage{multicol}
\usepackage{colordvi}
\usepackage[active]{srcltx} 
\usepackage{tikz}

\usepackage{xcolor}
\usepackage{graphicx}
\usepackage{shuffle}
\allowdisplaybreaks[4]
\usepackage{ifpdf}
\ifpdf
  \usepackage[colorlinks,final,backref=page,hyperindex]{hyperref}
\else
  \usepackage[colorlinks,final,backref=page,hyperindex,hypertex]{hyperref}
\fi
\usepackage{shuffle}

\newcommand{\tddeux}{\begin{picture}(5,5)(0,-1)
\put(3,0){\circle*{3}}
\thicklines
\put(3,0){\line(0,1){5}}
\put(3,5){\circle*{3}}
\end{picture}}

\newcommand{\tddeuxa}[2]{\begin{picture}(5,5)(0,-1)
\put(3,-1){\circle*{3}}
\thicklines
\put(3,-0.5){\line(0,1){7}}
\put(3,6){\circle*{3}}
\put(6,-3.8){\tiny #1}
\put(6,4.2){\tiny #2}
\end{picture}}

\newcommand{\tddeuxaa}[2]{\begin{picture}(5,5)(0,-1)
\put(3,-3.5){\circle*{3}}
\thicklines
\put(3,-3.5){\line(0,1){6}}
\put(3,3.5){\circle*{3}}
\put(6,-6){\tiny #1}
\put(6,1){\tiny #2}
\end{picture}}

\newcommand{\tdddeuxa}[3]{\begin{picture}(5,5)(0,-1)
\put(3,0){\circle*{3}}
\thicklines
\put(3,0){\line(0,1){6}}
\put(3,6){\circle*{3}}
\put(3,6){\line(0,1){6}}
\put(3,12){\circle*{3}}
\put(6,-5){\tiny #1}
\put(6,2.5){\tiny #2}
\put(6,8.5){\tiny #3}
\end{picture}}

\newcommand{\tdddeuxaa}[3]{\begin{picture}(5,5)(0,-1)
\put(3,-8){\circle*{3}}
\thicklines
\put(3,-8){\line(0,1){6}}
\put(3,-2){\circle*{3}}
\put(3,-2){\line(0,1){6}}
\put(3,4){\circle*{3}}
\put(6,-12){\tiny #1}
\put(6,-6){\tiny #2}
\put(6,0){\tiny #3}
\end{picture}}

\newcommand{\tdtroisuna}[3]{\begin{picture}(12,12)(-5,-1)
\put(3,0){\circle*{3}}
\put(-0.65,0){$\vee$}
\put(6,7){\circle*{3}}
\put(0,7){\circle*{3}}
\put(5,-2){\tiny #1}
\put(8,5){\tiny #2}
\put(-6,5){\tiny #3}
\end{picture}}

\newcommand{\tdtroisunab}[1]{\begin{picture}(5,5)(0,-1)
\put(3,0){\circle*{3}}
\put(5,-2){\tiny #1}
\end{picture}}

\newcommand{\tdtroisunac}[3]{\begin{picture}(5,5)(0,-1)
\put(3,-4){\circle*{3}}
\thicklines
\put(3,-4){\line(0,1){6}}
\put(3,2){\circle*{3}}
\put(3,2){\line(0,1){6}}
\put(3,8){\circle*{3}}
\put(5,-6){\tiny #1}
\put(5,0){\tiny #2}
\put(5,6){\tiny #3}
\end{picture}}

\newcommand{\tdtroisunad}[3]{\begin{picture}(12,12)(-5,-1)
\put(3,-2){\circle*{3}}
\put(-0.65,-2){$\vee$}
\put(6,5){\circle*{3}}
\put(0,5){\circle*{3}}
\put(5,-4){\tiny #1}
\put(8,3){\tiny #2}
\put(-6,3){\tiny #3}
\end{picture}}

\newcommand{\tdtroisunae}[2]{\begin{picture}(5,5)(0,-1)
\put(3,-2){\circle*{3}}
\thicklines
\put(3,-2){\line(0,1){7}}
\put(3,5){\circle*{3}}
\put(5,-4){\tiny #1}
\put(5,3){\tiny #2}
\end{picture}}

\newcommand{\tdtroisunaf}[3]{\begin{picture}(5,5)(0,-1)
\put(8,-2){\circle*{3}}
\thicklines
\put(8,-2){\line(0,1){7}}
\put(8,5){\circle*{3}}
\put(-1,1.5){\circle*{3}}
\put(10,-4){\tiny #1}
\put(10,3){\tiny #2}
\put(1,-0.5){\tiny #3}
\end{picture}}

\newcommand{\tdtroisunaa}[4]{\begin{picture}(12,12)(-5,-1)
\thicklines
\put(3,-4){\circle*{3}}
\put(-0.85,-3){$\vee$}
\put(6,3){\circle*{3}}
\put(0,3){\circle*{3}}
\put(6,10){\circle*{3}}
\put(5,-6){\tiny #1}
\put(8,1){\tiny #2}
\put(6,3){\line(0,1){7}}
\put(-6,1){\tiny #3}
\put(8,8){\tiny #4}
\end{picture}}

\newcommand{\tdtroa}[5]{\begin{picture}(5,5)(0,-1)
\thicklines
\put(-1,-2){\circle*{3}}
\put(-1,-2){\line(0,1){7}}
\put(-1,5){\circle*{3}}
\put(1,-4){\tiny #1}
\put(1,3){\tiny #2}

\put(8,-2){\circle*{3}}
\put(8,-2){\line(0,1){7}}
\put(8,5){\circle*{3}}
\put(10,-4){\tiny #3}
\put(10,3){\tiny #4}

\put(17,1.5){\circle*{3}}
\put(19,-0.5){\tiny #5}
\end{picture}}

\newcommand{\tdtrob}[4]{\begin{picture}(5,5)(0,-1)
\thicklines
\put(-1,1.5){\circle*{3}}
\put(1,-0.5){\tiny #1}

\put(8,-2){\circle*{3}}
\put(8,-2){\line(0,1){7}}
\put(8,5){\circle*{3}}
\put(10,-4){\tiny #2}
\put(10,3){\tiny #3}

\put(17,1.5){\circle*{3}}
\put(19,-0.5){\tiny #4}
\end{picture}}

\newcommand{\tdtroc}[4]{\begin{picture}(5,5)(0,-1)
\thicklines
\put(-1,-2){\circle*{3}}
\put(-1,-2){\line(0,1){7}}
\put(-1,5){\circle*{3}}
\put(1,-4){\tiny #1}
\put(1,3){\tiny #2}

\put(8,1.5){\circle*{3}}
\put(10,-0.5){\tiny #3}

\put(17,1.5){\circle*{3}}
\put(19,-0.5){\tiny #4}
\end{picture}}

\newcommand{\tdtrod}[3]{\begin{picture}(5,5)(0,-1)
\thicklines
\put(-1,-2){\circle*{3}}
\put(-1,-2){\line(0,1){7}}
\put(-1,5){\circle*{3}}
\put(1,-4){\tiny #1}
\put(1,3){\tiny #2}

\put(8,1.5){\circle*{3}}
\put(10,-0.5){\tiny #3}
\end{picture}}

\newcommand{\tdtroe}[3]{\begin{picture}(5,5)(0,-1)
\put(-1,0){\circle*{3}}
\put(1,-2){\tiny #1}

\put(8,0){\circle*{3}}
\put(10,-2){\tiny #2}

\put(17,0){\circle*{3}}
\put(19,-2){\tiny #3}
\end{picture}}

\newcommand{\tdtrof}[2]{\begin{picture}(5,5)(0,-1)
\put(-1,0){\circle*{3}}
\put(1,-2){\tiny #1}

\put(8,0){\circle*{3}}
\put(10,-2){\tiny #2}
\end{picture}}

\newcommand{\tdtrog}[4]{\begin{picture}(5,5)(0,-1)
\thicklines
\put(-1,-2){\circle*{3}}
\put(-1,-2){\line(0,1){7}}
\put(-1,5){\circle*{3}}
\put(1,-4){\tiny #1}
\put(1,3){\tiny #2}

\put(8,-2){\circle*{3}}
\put(8,-2){\line(0,1){7}}
\put(8,5){\circle*{3}}
\put(10,-4){\tiny #3}
\put(10,3){\tiny #4}
\end{picture}}

\newtheorem{theorem}{Theorem}[section]
\newtheorem{proposition}[theorem]{Proposition}
\newtheorem{lemma}[theorem]{Lemma}

\newtheorem{prop-def}{Proposition-Definition}[section]
\newtheorem{coro-def}{Corollary-Definition}[section]

\theoremstyle{definition}
\newtheorem{definition}[theorem]{Definition}
\newtheorem{remark}[theorem]{Remark}

\newcommand{\nc}{\newcommand}

\newcommand{\delete}[1]{}

\nc{\mlabel}[1]{\label{#1}}
\nc{\mcite}[1]{\cite{#1}}
\nc{\mref}[1]{\ref{#1}}
\nc{\meqref}[1]{\eqref{#1}}
\nc{\mbibitem}[1]{\bibitem{#1}}

\delete{
\nc{\mlabel}[1]{\label{#1}{\hfill \hspace{1cm}{\bf{{\ }\hfill(#1)}}}}
\nc{\mcite}[1]{\cite{#1}{{\bf{{\ }(#1)}}}}
\nc{\mref}[1]{\ref{#1}{{\bf{{\ }(#1)}}}}
\nc{\meqref}[1]{\eqref{#1}{{\bf{{\ }(#1)}}}}
\nc{\mbibitem}[1]{\bibitem[\bf #1]{#1}}
}

\nc{\calT}{\mathcal{T}} \nc{\calF}{\mathcal{F}} \nc{\f}{\calF  }
\nc{\fn}{\calF^{n}} \nc{\fsn}{\calF^{\leq n}} \nc{\fsN}{\calF^{\leq N}}
\nc{\hck}{\mathcal{H}_{\rm CK}}
\nc{\hgl}{\mathcal{H}_{\rm GL}}
\nc{\RR}{\mathbb{R}} \nc{\ZZ}{\mathbb{Z}}  \nc{\etree}{\mathbbm{1}} \nc{\conc}{\cdot}
\nc{\hckn}{\hck^n} \nc{\hgln}{\hgl^n}
\nc{\hcksn}{\hck^{\leq n}} \nc{\hglsn}{\hgl^{\leq n}} \nc{\hcksN}{\hck^{\leq N-1}} \nc{\hglsN}{\hgl^{\leq N}}
\nc{\X}{{\bf X}} \nc{\brp}{{\bf BRP}_\alpha^{N}} \nc{\brpt}{{\bf BRP}_\alpha^{3}}
\nc{\cbrpx}{\mathcal{D}^{N}_{\X;\alpha }}
\nc{\cbrptx}{\mathcal{D}^{N}_{\tilde{X};\alpha }}
\nc{\cbrpxt}{\mathcal{D}^{3}_{\X;\alpha }}
\nc{\cbrpxs}{\mathcal{D}^{2}_{\X;\alpha }}
\nc{\cbrptxt}{\mathcal{D}^{3}_{\tilde{X};\alpha }}
\nc{\Y}{{\bf Y}} \nc{\Z}{{\bf Z}} \nc{\W}{{\mathcal{Y}}}
\nc{\h}{{\bf H}}  \nc{\U}{{\bf U}} \nc{\bfone}{{\bf 1}}
\nc{\con}[2]{\ell_{#1}(#2)} \nc{\cont}[3]{\ell_{#1}^{#3}(#2)}
\nc{\fan}[1]{|||#1|||}
\nc{\hcksth}{\hck^{\leq 3}} \nc{\hglsth}{\hgl^{\leq 3}}
\nc{\ckr}{\Delta_{{\rm CK}}^{{\rm red}}}

\nc{\R}{\resizebox{0.8\width}{0.8\height}}
\nc{\Int}{\displaystyle\int}
\nc{\hm}{\mathcal{H}_{\text {MKW}}}
\nc{\he}{\mathcal{H}_{\text {MKW}}^{g}}
\nc{\hmn}{\hm^n} \nc{\hen}{\mathcal{H}_{\text {MKW}}^{*n}}
\nc{\hmsn}{\hm^{\leq N}} \nc{\hesn}{\mathcal{H}_{\text {MKW}}^{*\leq N}} \nc{\hmsN}{\hm^{\leq N-1}} \nc{\hesN}{\mathcal{H}_{\text {MKW}}^{*\leq N}}
\nc{\pbrp}{{{\bf PBRP} }_\alpha^{N}}
\nc{\pbrps}{{{\bf PBRP}}_\alpha^{2}}
\nc{\pbrpt}{{{\bf PBRP}}_\alpha^{3}}
\nc{\hmsnd}{\hm^{\leq 1}} \nc{\hesnd}{\mathcal{H}_{\text {MKW}}^{*\leq 2}}
\nc{\hmsth}{\hm^{\leq 2}} \nc{\hesth}{\mathcal{H}_{\text {MKW}}^{*\leq 3}}
\nc{\EX}{\widehat{\bf X}}
\nc{\hmr}{\Delta_{{\rm \text {MKW}}}^{{\rm red}}}
\nc{\hmnn}{\hm^n} \nc{\henn}{\mathcal{H}_{\text {MKW}}^{*n}}
\nc{\lbar}{\hat}
\nc{\cbar}{\overline}

\nc{\xing}[1]{\textcolor{blue}{Xing:#1}}
\nc{\Dominique}[1]{\textcolor{blue}{Dominique: #1}}
\nc{\blue}[1]{\textcolor{blue}{#1}}
\nc{\revise}[1]{\textcolor{red}{#1}}
\nc{\red}[1]{\textcolor{red}{#1}}
\numberwithin{equation}{section}

\begin{document}
\title[It\^o formula for planarly branched rough paths]{It\^o formula for planarly branched rough paths}
%
%
\author{Nannan Li}
\address{School of Mathematics and Statistics, Lanzhou University
Lanzhou, 730000, China
}
\email{linn2024@lzu.edu.cn}

\author{Xing Gao$^{*}$}\thanks{*Corresponding author}
\address{School of Mathematics and Statistics, Lanzhou University
Lanzhou, 730000, China;
Gansu Provincial Research Center for Basic Disciplines of Mathematics
and Statistics, Lanzhou, 730070, China
}
\email{gaoxing@lzu.edu.cn}

\begin{abstract}
The It\^o formula is a cornerstone of stochastic calculus, providing a fundamental change-of-variables formula for stochastic processes.
In rough path theory, It\^o formulas have been proved for rough paths with roughness $\frac{1}{3}< \alpha \le \frac{1}{2}$ and branched rough paths with roughness $0< \alpha \le1$.
Planarly branched rough paths provide a natural framework for studying rough differential equations on homogeneous spaces, extending the theory of branched rough paths through the Munthe–Kaas–Wright Hopf algebra.
In the present paper, we prove the It\^o formula for planarly branched rough paths with roughness $\frac{1}{4}< \alpha \le \frac{1}{2}$.
\end{abstract}

\makeatletter
\@namedef{subjclassname@2020}{\textup{2020} Mathematics Subject Classification}
\makeatother
\subjclass[2020]{
60L20, 
60H99, 
34K50, 
37H10,  
05C05.   
}

\keywords{It\^o formula, planarly branched rough path, controlled planarly branched rough path.}

\maketitle

\tableofcontents

\setcounter{section}{0}
\allowdisplaybreaks

\section{Introduction}
In this paper, the It\^o formula is generalized from rough paths and branched rough paths to planarly branched rough paths with roughness $\frac{1}{4}< \alpha \leq \frac{1}{2}$.

\subsection{Rough paths and controlled rough paths}
Rough path theory~\cite{L98} is concerned with the rough differential equations (RDEs) of the form
\begin{equation}
\left\{
\begin{array}{rll}
dY_{t}&={f}(Y_{t})\cdot dX_t=\sum\limits_{i=1}^{d}{f_i}(Y_{t})\, dX_{t}^{i}, \quad \forall t\in [0, T],\\
Y_0 &= \xi  ,
\end{array}\right.
\mlabel{ivp}
\end{equation}
where $X=(X^1,\ldots, X^d):[0,T]\to \mathbb R^d$ is of $\alpha$-H\"{o}lder regularity for $0<\alpha\le 1$, and $f_1,\ldots, f_d$ are smooth vector fields on $\mathbb R^n$. 
It is evident that the interpretation of the integral of~\eqref{ivp} determines the solution Y, and this integral can be elegantly encoded using the notion of rough path. 
To give a rigorous meaning to \eqref{ivp}, we enhance $X$ to a rough path $\X$ and formulate \eqref{ivp} in the space of $\X$-controlled rough paths as~\cite{Gub04}
\begin{equation}
\left\{
\begin{array}{rll}
d\Y_{t}&={f}(\Y_{t})\cdot dX_t=\sum\limits_{i=1}^{d}{f_i}(\Y_{t}) dX_{t}^{i}, \quad \forall t\in [0, T],\\
Y_0 &= \xi  ,
\end{array}\right.
\mlabel{ivpp}
\end{equation}
where $\Y=(\mathbf Y^1,\ldots,\mathbf Y^n)$ is an $\X$-controlled rough path above $Y$ in the sense of Definition~\ref{defn:cbrp}.
Then \eqref{ivpp} should be understood as its integral version \eqref{eq:BRDE}, where the integral is the rough integral given in \eqref{eq:inte2}.
A two-parameter path $\mathbf X$ is called a rough path if $\X$ satisfies Chen's lemma $\mathbf X_{st}=\mathbf X_{su}\star\mathbf X_{ut}$ and suitable estimates on the shuffle Hopf algebra.

Gubinelli~\cite{Gub10} introduced branched rough paths by replacing the shuffle Hopf algebra of words underlying geometric rough paths with the Butcher-Connes-Kreimer (BCK) Hopf algebra of rooted forests.
The concept of planarly branched rough path~\cite{CFMK,KL23}, which was used to solve RDEs on homogeneous spaces~\cite{CFMK}, is obtained by replacing the BCK Hopf algebra of rooted forests by the Munthe-Kaas-Wright (MKW) Hopf algebra of planarly rooted forests. 
The MKW Hopf algebra is the Hopf algebra of Lie group integrators introduced by Munthe-Kaas and Wright in~\cite{MKW}.
The MKW (resp. BCK, resp. shuffle) Hopf algebra is the foundation of planarly branched rough paths (resp. branched rough paths, resp. geometric rough paths), and the primitive elements of its graded dual Hopf algebra constitute a free post-Lie (resp. pre-Lie, resp. Lie) algebra. 
Therefore, in the sense that post-Lie algebras are generalizations of both pre-Lie algebras and Lie algebras, planarly branched rough paths are generalizations of both branched rough paths and geometric rough paths.
From a stochastic perspective, the planar structure provides a natural algebraic framework for studying rough differential equations on homogeneous spaces driven by suitably enhanced stochastic processes~\cite{CFMK}.

Gubinelli~\cite{Gub04} first introduced controlled rough paths, in the shuffle version~\cite{Gub04} as well as in the branched version~\cite{Gub10}. 
The corresponding theory for planarly branched rough paths was developed in~\cite{GLM24}.
Namely, an $\mathbf X$-controlled rough path is a map $\mathbf Y$ from $[0,T]$ into the (suitably truncated) Hopf algebra, such that
\[\langle \tau, \mathbf Y_t\rangle=\langle \mathbf X_{st}\star \tau,\mathbf Y_s\rangle + \hbox{ \sl small \rm remainder},\]
whose precise meaning is described in Definition~\ref{defn:cbrp} below. 

\subsection{It\^o formula for rough paths}
It\^o formula~\cite{Ito} serves as the foundation for stochastic analysis and connects stochastic processes and partial differential equations. For a twice continuously differentiable function $f$, the classical It\^o formula~\cite{Ito} reads:
\begin{equation}\mlabel{kvp}
f (X_t)= f (X_s)+\int_s^t f'(X_r)\, dX_r+\frac{1}{2}\int_s^t f''(X_r)\, dr,
\end{equation}
where $X$ is a real-valued Brownian motion and the integral $\int f'(X_r)\, dX_r$ is the It\^o integral. The It\^o formula has undergone numerous extensions since its introduction, most notably to semimartingales. 
Generalizations of the It\^o formula have also been developed for various stochastic processes, some of which are not directly covered by classical semimartingale calculus.
These include 4-stable processes~\cite{BM96}, Brownian motions~\cite{IK07}, fractional Brownian motions~\cite{GNRV05,GRV03}, finite $p$-variation processes~\cite{ER03} and solutions to stochastic partial differential equations~\cite{BS10,HN21}.

Some of these examples, such as Brownian motion and fractional Brownian motion, have sample paths that are almost surely H\"older continuous with a positive exponent and can be studied within the rough path framework.
Friz and Hairer~\cite{FH20} built the It\^o formula for rough paths with roughness $\frac{1}{3}<\alpha\le \frac{1}{2} $. Afterwards, Kelly proved the It\^o formula for branched rough paths with roughness $0<\alpha\le 1$ in~\cite{Kel}. In~\cite{FZ18}, Friz and Zhang gave the It\^o formula for rough paths with jumps and roughness $\frac{1}{3}< \alpha\le \frac{1}{2}$.
More recently, Bellingeri, Ferrucci and Tapia~\cite{BFT26} established an It\^o formula for branched rough paths of arbitrary finite roughness based on the BCK Hopf algebra, using its primitive elements to characterize the correction terms.\vskip 0.1in

\subsection{Munthe-Kaas-Wright Hopf algebra and results of the paper}
Classical Butcher series describe numerical integration methods for ordinary differential equations in Euclidean spaces. 
The extension of numerical integration methods to manifolds~\cite{CR93} motivated the development of Lie-group integration techniques~\cite{IMK00} and Lie-Butcher series~\cite{MK95}, the latter generalizing classical Butcher series.
The algebraic theory of Butcher series is closely related to the BCK Hopf algebra of non-planar rooted forests~\cite{CK98} and to pre-Lie algebraic structures on rooted trees~\cite{CEM11}.
Guin and Oudom's work~\cite{OG08} further highlights the complex relationship between the algebraic framework of Butcher series and the features of non-planar rooted trees by demonstrating how the free pre-Lie algebra may be used to derive the BCK Hopf algebra.

On the other hand, the theory of Lie-Butcher series necessitates a shift from non-planar rooted trees to planar counterparts.
In this transition, the free post-Lie algebra defined on planar rooted trees~\cite{MKL13} was used in place of the free pre-Lie algebra on non-planar trees.
It is possible to derive the MKW Hopf algebra by applying the Guin-Oudom result to post-Lie algebras~\cite{AFM, ELM15}.
In~\cite{ER25}, Ebrahimi-Fard and Rahm summarized the important role played by the MKW Hopf algebra in the context of planarly branched rough paths. 
There they also investigated the geometric embedding for planar regularity structures. 
Quite recently, we proved the universal limit theorem in the framework of planarly branched rough paths with roughness $\frac{1}{4}< \alpha \leq \frac{1}{3}$ in~\cite{GLM24}.

In this paper, we derive It\^o formulas for planarly branched rough paths in the two roughness regimes $\frac{1}{3}<\alpha\leq\frac{1}{2}$ and $\frac{1}{4}<\alpha\leq\frac{1}{3}$, corresponding to truncation levels $N=2$ and $N=3$.

Compared to the recent work of Bellingeri, Ferrucci and Tapia~\cite{BFT26}, which establishes a general It\^o formula for branched rough paths within the BCK Hopf algebra framework, the present work develops It\^o formulas directly within the MKW Hopf algebra framework.
In particular, the proof of their main It\^o formula~\cite[Theorem 2.14]{BFT26} relies on the product and coproduct of the BCK Hopf algebra, namely, the commutative product of non-planar rooted forests and the coproduct defined by admissible cuts.
In contrast, the product in the MKW Hopf algebra underlying planarly branched rough paths is the shuffle product of ordered forests, whereas its coproduct is defined by left-admissible cuts.
These algebraic structures lead to different relations among rough path components and different expansions involving ordered trees and forests. 
Consequently, the algebraic constructions and expansions underlying the proof of their main It\^o formula in the BCK Hopf algebra framework cannot be directly transferred to the MKW Hopf algebra framework.
To address these differences, we establish controlled rough path expansions and RDE Taylor expansions directly within the MKW Hopf algebra framework.
We also construct planar bracket extensions compatible with the algebraic relations of the MKW Hopf algebra and satisfying the required mixed H\"older estimates. 

The main difficulties in extending these results to the regime $\alpha\leq\frac{1}{4}$ are
\begin{enumerate}
\item to prove that compositions of smooth functions with controlled planarly branched rough paths are controlled, and

\item to prove that there is a unique solution of the RDE~(\mref{ivp}) in the sense of Definition~\mref{defn:bRDE} below,
\end{enumerate}
which will be considered in the future work. \vskip 0.1in

{\bf Outline.} The paper is organised as follows. In Section~\ref{ss:sec2}, we first recall the notions of planarly branched rough path, controlled planarly branched rough path and rough integral of a controlled planarly branched rough path against a planarly branched rough path.
Then we obtain new controlled planarly branched rough paths via compositions of smooth functions with controlled planarly branched rough paths  (Propositions~\ref{pp:regu1} and~\ref{pp:regu2}).
Finally, we obtain a kind of Taylor expansion based on planarly branched rough paths (Proposition~\ref{pp:RDE1}).
The goal of Section~\ref{sec:3} is to prove respectively It\^o formulas for $\alpha$-H\"{o}lder planarly branched rough paths with roughness $\alpha\in (\frac{1}{3}, \frac{1}{2}]$ and $\alpha\in (\frac{1}{4}, \frac{1}{3}]$, both of which are divided into the simple case $F(X)$ (Theorems~\mref{thm:ito1} and~\mref{thm:ito2}) and the general case $F(Y)$ (Theorems~\mref{thm:yitoa} and~\mref{thm:yitob}).
 \vskip 0.1in

{\bf Notation.} Throughout this paper, we fix two positive integers $d$ and $n$. The first integer is the number of vector fields driving~\eqref{ivp}, and the second integer is the dimension of the vector space in which the solutions of \eqref{ivp} take their values. Let $A:=\{1, \ldots, d\}$ be a finite set used as a decoration set. Denote by $\RR$ the field of real numbers, which will serve as the base field of all vector spaces, algebras, tensor product, as well as linear maps.
For any map $Y:[0,T]\rightarrow \RR^n$, we define $\delta Y_{s,t}:=Y_t-Y_s $. Let $o(|t-s|)$ (resp. $O(|t-s|)$) be infinitesimal of higher (resp. higher or the same) order of $|t-s|$.

\section{Rough integrals of planarly branched rough paths}\label{ss:sec2}
In this section, we first recall the main concepts studied in this paper, including planarly branched rough paths, controlled planarly branched rough paths, as well as rough integral of a controlled planarly branched rough path against a planarly branched rough path.
Then we obtain new controlled planarly branched rough paths in terms of compositions of smooth functions, and a kind of Taylor expansion based on planarly branched rough paths.

\subsection{Planarly branched rough paths and controlled planarly branched rough paths}\label{sec:2.1}
Much in this subsection traces its roots to~\cite{CFMK, GLM24}. Denote by $\calT$ (resp. $\f$) the set of $A$-decorated planarly rooted trees (resp. forests).
For $\tau\in \f$, define $|\tau|$ to be the number of vertices of $\tau$.
There is a unique forest $\etree$ with $|\etree|=0$, called the empty forest.
Denote respectively by
\begin{align*}
\hm:=&\ (\mathbb{R} \calF, \shuffle, \etree, \Delta_{\text {MKW}},\etree^*)=\bigoplus_{k\geq 0} \mathcal{H}_{\text {MKW}}^k,\\
\he:=&\ (\mathbb{R} \f,\star, \etree, \Delta_{\shuffle},\etree^*)= \bigoplus_{k\geq 0}
(\mathcal{H}_{\text {MKW}}^{k})^*,
\end{align*}
the MKW graded Hopf algebra~\cite{MKW} and its graded dual Hopf algebra under the pairing
$
\langle\tau,\tau'\rangle=\delta_\tau^{\tau'},
$
where $\shuffle$ is the shuffle product, $\delta $ is the Kronecker delta and
$${\mathcal{F}}^{k}:=\{\tau\in \f \mid |\tau| = k \}, \quad \mathcal{H}_{\text {MKW}}^k:= \mathbb{R} {\mathcal{F}}^{k}=: (\mathcal{H}_{\text {MKW}}^{k})^*, \quad \forall k, N\geq 0.$$
Here since $\dim \mathcal{H}_{\text {MKW}}^{k}$ is finite, we can view ${\mathcal{F}}^{k}$ as a basis of $(\mathcal{H}_{\text {MKW}}^{k})^*$. Set $$\mathcal{T}^{k}:=\mathcal{F}^{k}\cap \mathcal{T},\quad {\mathcal{F}}^{\le N}:=\{\tau\in \f \mid |\tau| \leq N \}, \quad {\mathcal{T}}^{\le N}:={\mathcal{F}}^{\le N}\cap \mathcal{T}, \quad \forall k, N\geq 0.$$

The coproduct $\Delta_{\text {MKW}}$ and the multiplication $\star$ are characterized by the left admissible cut and the left grafting, respectively. See~\cite[Section~3]{ER25} and~\cite[Section~3.1]{MKW}. We recall here the coproduct $\Delta_{\text {MKW}}$ in detail for later use. Let $\tau$ be an $A$-decorated planarly rooted tree in $\calT$. A subset $c$ of edges in $\tau$ is called {\bf a left-admissible cut} if
\begin{enumerate}
\item For every edge $e$ in $c$, every other edge outgoing from the same vertex as $e$ and to the left of $e$ in the planar embedding, is also in $c$.

\item Every path from the root of $\tau$ to any of its leaves contains at most one edge in $c$.
\end{enumerate}
Let $c$ be a left-admissible cut of $\tau$ and remove the edges in $c$ from $\tau$.
Denote by $T^c(\tau)$ the trunk part, which is the connected component containing the root.
Let $P^c(\tau)$ be the pruned part, which is obtained by keeping the planar order of those subtrees that are cut off from the same vertex and then shuffling together the resulting planarly rooted forests that are cut off from
different vertices.
The coproduct $\Delta_{\text {MKW}}$ is described as
\begin{align*}
\Delta_{\text {MKW}}(\tau):=&\ \sum_{c \text{\ left-admissible cut}}P^c(\tau)\otimes T^c(\tau),\quad \forall \tau \in \calT,\\
\Delta_{\text {MKW}}(\ell ):=&\ (\text{id}\otimes B_{-})(\Delta_{\text {MKW}}(B_{+}(\ell ))-B_{+}(\ell )\otimes \etree),\quad  \forall \ell  \in \calF,
\end{align*}
where $B_{+}$ is the usual grafting operator, grafting all planarly rooted trees in the input planarly rooted forest onto a new root (no need decoration) and $B_{-}$ is the inverse of $B_{+}$. We expose two examples~\cite{ER25} for better understanding:
\begin{align*}
\Delta_{\text {MKW}}(\ \tdtroisunaa{$a$}{$c$}{$b$}{$d$}\ \ \ )=&\ \etree\otimes\, \tdtroisunaa{$a$}{$c$}{$b$}{$d$}\ \ +\tdtroisunab{$b$}\ \,\otimes \tdtroisunac{$a$}{$c$}{$d$}\ \ +\tdtroisunab{$d$}\ \,\otimes\, \tdtroisunad{$a$}{$c$}{$b$}\ \,+\tdtroisunab{$b$}\ \,\shuffle \tdtroisunab{$d$}\ \,\otimes \tdtroisunae{$a$}{$c$}\ \,+\ \tdtroisunaf{$c$}{$d$}{$b$}\ \ \ \,\otimes \tdtroisunab{$a$}\ +\,\tdtroisunaa{$a$}{$c$}{$b$}{$d$}\ \ \otimes\etree,\\
\Delta_{\text {MKW}}(\ \ \tdtroa{$a$}{$b$}{$c$}{$d$}{$f$}\ \ \ \ \ \ \ )=&\ \etree\otimes\,\, \tdtroa{$a$}{$b$}{$c$}{$d$}{$f$}\ \ \ \ \ \ \ +\tdtroisunab{$b$}\ \,\otimes\  \tdtrob{$a$}{$c$}{$d$}{$f$}\ \ \ \ \ \ \, +\tdtroisunab{$d$}\ \,\otimes\ \tdtroc{$a$}{$b$}{$c$}{$f$}\ \ \ \ \ \ \, +\tdtroisunae{$a$}{$b$}\ \,\otimes\ \tdtrod{$c$}{$d$}{$f$}\ \ \, \ +\tdtroisunab{$b$}\ \,\shuffle \tdtroisunab{$d$}\ \,\,\otimes \ \tdtroe{$a$}{$c$}{$f$}\\
&\ +\tdtroisunae{$a$}{$b$}\ \,\shuffle \tdtroisunab{$d$}\ \,\,\otimes\, \, \tdtrof{$c$}{$f$}\ \ \ \ +\ \tdtrog{$a$}{$b$}{$c$}{$d$}\ \ \ \ \otimes\tdtroisunab{$f$}\ \,+\ \tdtroa{$a$}{$b$}{$c$}{$d$}{$f$}\ \ \ \ \ \ \ \otimes\etree.
\end{align*}

Both $$\hmsn := \bigoplus_{0\leq k\leq N} \hm^k, \quad (\mathcal{H}_{\text {MKW}}^{g})^{\le N} := \bigoplus_{0\leq k\leq N} (\mathcal{H}_{\text {MKW}}^{k})^*$$ are endowed with a structure of connected graded (and finite-dimensional) algebra and  coalgebra as follows: the vector subspace $I_N:= \RR\{\tau\in \f \mid |\tau| > N \}$ of $\hm$ is a graded ideal (but not a bi-ideal), hence $\hm/I_N$ is a graded algebra.
On the other hand, the restriction of the projection $\pi_N:\hm\to\hskip -9pt\to\hm/I_N$ to the subcoalgebra $\hmsn$ is an isomorphism of graded vector spaces.
The graded algebra structure of $\hm/I_N$ can therefore be transported on $\hmsn$, making  it both an algebra and a coalgebra, denoted by
$$\hmsn:= ( \mathbb{R} \fsN, \shuffle, \etree, \Delta_{\text {MKW}},\etree^*)$$
by a slight abuse of notations. Since $\dim \hmsn$ is finite, we also have the dual algebra/coalgebra
$$(\mathcal{H}_{\text {MKW}}^{g})^{\le N} := (\mathcal{H}_{\text {MKW}}^{\le N})^g=(\mathbb{R} \fsN,\star, \etree^*, \Delta_{\shuffle},\etree)$$ under the above pairing. The following is
 the concept of planarly branched rough path.

\begin{definition}~\cite[Definition 4.1]{CFMK}
Let $\alpha \in(0,1]$ and $N= \lfloor \frac{1}{\alpha} \rfloor$. An {\bf $\alpha$-H\"older planarly branched rough path} is a map $\X:[0,T]^2\rightarrow (\mathcal{H}_{\text {MKW}}^{g})^{\le N}$ such that
\begin{enumerate}
\item  $\X_{s,u} \star \X_{u,t} =\X_{s,t}$;\mlabel{it:bitem1}

\item $\langle \X_{s,t}, \sigma \shuffle \tau\rangle  = \langle \X_{s,t}, \sigma\rangle\langle \X_{s,t}, \tau\rangle $, for $|\sigma|+|\tau|\le N$;\mlabel{it:bitem2}

\item $\sup\limits_{s\ne t\in [ 0,T ] } \frac{| \left \langle \X_{s,t}, \tau \right \rangle  |}{| t-s | ^{| \tau |\alpha  } } <\infty$,\quad $\forall \tau\in \fsN$ and $s, t, u\in \left[ 0, T\right] $. \mlabel{it:bitem3}
\end{enumerate}
Further for a path $X=(X^1,\ \ldots,X^d):[0,T]\rightarrow \mathbb{R}^d$, if
\begin{equation}
\langle \X_{s,t} , \bullet_i \rangle=X_{t}^i - X_{s}^i, \quad \forall i=1,\ldots, d,
\mlabel{eq:above}
\end{equation}
then we call $\X$ a planarly branched rough path above $X$.
\mlabel{def:pbrp}
\end{definition}

Let $\pbrp $ be the set of $\alpha$-H\"{o}lder planarly branched rough paths. Next, we recall the concept of controlled planarly branched rough path.

\begin{definition}~\cite[Definition 2.2]{GLM24}
Let $\X\in\pbrp $. A path $\Y: [0,T]\rightarrow \hmsN $ is called an  {\bf $\X$-controlled planarly branched rough path} if \begin{equation}
|R\Y^{\tau}_{s,t}|=O(|t-s|^{(N-|\tau|)\alpha}),
\mlabel{eq:remainry0}
\end{equation}
where
\begin{equation}
R\Y^{\tau}_{s,t}:= \left \langle \tau,\Y_t  \right \rangle- \left \langle \X_{s,t}\star \tau,\Y_s \right \rangle,\quad \forall \tau\in \mathcal{F}^{\le N-1}.
\mlabel{eq:remainry}
\end{equation}
Further for a path $Y:[0,T]\rightarrow \mathbb{R}$, if
$\langle \etree, \Y_{t} \rangle=Y_{t},$
then we call $\Y$ above $Y$.
\mlabel{defn:cbrp}
\end{definition}

\begin{remark}
One can adapt Definition~\ref{defn:cbrp} componentwise to
$$\Y=(\Y^1, \ldots, \Y^n): [0,T]\rightarrow (\hmsN)^n,\quad t\mapsto \Y_t=(\Y^1_t, \ldots, \Y^n_t), $$
where $(\hmsN)^n$ denotes the $n$-th cartesian power of $\hmsN$. Notice that
\begin{equation*}
\langle \tau, \Y_{t} \rangle\in \RR^n, \quad \forall \tau\in \mathcal{F}^{\le N-1}.
\end{equation*}
For $\X\in \pbrp$, $\X$-controlled planarly branched rough paths form a space~\mcite{GLM24}, denoted by $\cbrpx$.
\mlabel{rk:gcont}
\end{remark}

\subsection{Rough integrals}\label{sec:2.2}
We are going to recall the concept of rough integral, based on planarly branched rough path and divided into the cases of $\alpha \in  ( \frac{1}{4}, \frac{1}{3}]$ and $\alpha \in  ( \frac{1}{3}, \frac{1}{2}]$.
{\em In the sequel} of the paper, we always let
\begin{equation*}
X=(X^1,\ldots,X^d):[0,T]\rightarrow \mathbb{R}^d,\quad Y:[0,T]\rightarrow \mathbb{R}^n,\quad f=(f_1,\ldots,f_d):\mathbb{R}^n \rightarrow  (\mathbb{R}^n)^d,
\end{equation*}
with each $f_i:\mathbb{R}^n \rightarrow\mathbb{R}^n $ smooth.
Notice that $\mathcal{F}^{\leq 1} =\{\etree, \bullet_i \mid i\in A \}$ and
\begin{equation}
\mathcal{F}^{\leq 2} =\{\etree, \bullet_i, \bullet_i\bullet_j, \tddeuxa{$i$}{$j$}\ \,\mid i,j\in A \}, \ \ \ \  \mathcal{F}^{\leq 3}=\{\etree, \bullet_i, \bullet_i\bullet_j, \tddeuxa{$i$}{$j$}\ \,, \bullet_i\tddeuxa{$j$}{$k$}\ \,, \tddeuxa{$i$}{$j$}\ \,\bullet_k, \bullet_i\bullet_j\bullet_k, \tdtroisuna{$i$}{$j$}{$k$}\ \,, \tdddeuxa{$i$}{$j$}{$k$}\ \, \mid i,j,k\in A\}.
\mlabel{eq:leq23}
\end{equation}
\vspace{0.1pt}

\begin{lemma}~\cite[Theorem 3.2]{GLM24}
Let $\alpha \in  ( \frac{1}{4}, \frac{1}{3}]$ and $N=3$. Let $\X\in \pbrp$ above $X$ and $\Y\in \cbrpx$ above Y. Define	
\begin{equation*}
\widetilde{Y}_{s,t} :=\sum_{\tau\in \mathcal{F}^{\leq 2}}\langle \tau, \Y_{s} \rangle\langle \X_{s,t}, [\tau]_i \rangle.
\end{equation*}
Then there is a unique function I denoted by $\int_{0}^{\cdot} Y_rdX_r^i$ such that
\begin{equation}
\Big|\R{$\Int_s^t$} Y_rdX_{r}^{i}-\sum_{\tau\in \mathcal{F}^{\leq 2}}\langle \tau, \Y_{s} \rangle\langle \X_{s,t}, [\tau]_i \rangle\Big|= O(| t-s | ^{4\alpha} )
\mlabel{eq:inte1-}
\end{equation}
and
\begin{equation}
I_t:=\R{$\Int_0^t$} Y_rdX_{r}^{i}=\lim_{|\pi | \to 0} \sum_{[t_i, t_{i+1}]\in \pi}\sum_{\tau\in \mathcal{F}^{\leq 2}}\langle \tau, \Y_{t_i} \rangle\langle \X_{t_i,t_{i+1}}, [\tau]_i \rangle \in \RR^n,
\mlabel{eq:intee1}
\end{equation}
where $\pi$ is an arbitrary partition of \,$[0, T]$. We call $\R{$\Int_0^t$} Y_rdX_{r}^{i}$ the {\bf rough integral} of $Y$ against $X^i$.
\mlabel{thm:inte1}
\end{lemma}

For the case of $\alpha \in  ( \frac{1}{3}, \frac{1}{2}]$, we have
\begin{lemma}
Let $\alpha \in  ( \frac{1}{3}, \frac{1}{2}]$ and $N=2 $. Let $\X\in \pbrps$ above $X$ and $\Y\in \cbrpxs$ above Y. Define	
\begin{equation*}
\widetilde{Y}_{s,t} :=\sum_{\tau\in \mathcal{F}^{\leq 1}}\langle \tau, \Y_{s} \rangle\langle \X_{s,t}, [\tau]_i \rangle.
\end{equation*}
Then there is a unique function I denoted by $\int_{0}^{\cdot} Y_rdX_r^i$ such that
\begin{equation}
\Big|\R{$\Int_s^t$} Y_rdX_{r}^{i}-\sum_{\tau\in \mathcal{F}^{\leq 1}}\langle \tau, \Y_{s} \rangle\langle \X_{s,t}, [\tau]_i \rangle\Big|= O(| t-s | ^{3\alpha} )
\mlabel{eq:inte2-}
\end{equation}
and
\begin{equation}
I_t:=\R{$\Int_0^t$} Y_rdX_{r}^{i}=\lim_{|\pi | \to 0} \sum_{[t_i, t_{i+1}]\in \pi}\sum_{\tau\in \mathcal{F}^{\leq 1}}\langle \tau, \Y_{t_i} \rangle\langle \X_{t_i,t_{i+1}}, [\tau]_i \rangle \in \RR^n,
\mlabel{eq:inte4}
\end{equation}
where $\pi$ is an arbitrary partition of $[0, T]$. We call $\R{$\Int_0^t$} Y_rdX_{r}^{i}$ the {\bf rough integral} of $Y$ against $X^i$.
\mlabel{thm:inte2}
\end{lemma}

\begin{proof}
The proof is similar and easier to the one of Lemma~\ref{thm:inte1}.
\end{proof}

The following result gives a map which sends a controlled planarly branched rough path $\Y$ to another one $\R{$\Int_0^\bullet$} \Y_rdX_{r}^{i}$.

\begin{lemma}
Let $\alpha \in (\frac{1}{4}, \frac{1}{2}]$ and $N= \lfloor \frac{1}{\alpha} \rfloor$. Let $\X\in \pbrp$ above $X$ and $\Y\in \cbrpx$ above Y. Define
\begin{equation*}
\R{$\Int_0^\bullet$} \Y_rdX_{r}^{i}:[0,T]\longrightarrow (\mathcal{H}_{\mathrm{MKW}}^{\le N-1})^n, \quad t\mapsto \R{$\Int_0^t$} \Y_rdX_{r}^{i}
\end{equation*}
by setting
\begin{equation}
\begin{aligned}
&\ \langle\etree,\R{$\Int_0^t$}\Y_{r}dX_{r}^{i}\rangle:=\R{$\Int_0^t$}Y_{r}dX_{r}^{i}, \quad
\langle [\tau]_{i},\R{$\Int_0^t$}\Y_{r} dX_{r}^{i}\rangle:=\ \langle \tau,\Y_{t}\rangle, \quad \forall \tau\in \mathcal{F}^{\leq N-2},\\
&\ \langle \tau,\R{$\Int_0^t$}\Y_{r} dX_{r}^{i}\rangle:= 0,\quad \text{otherwise.}
\end{aligned}
\mlabel{eq:intek1}
\end{equation}
Then $\R{$\Int_0^\bullet$} \Y_rdX_{r}^{i}$ is an $\X$-controlled planarly branched rough path above $\R{$\Int_0^\bullet$}Y_{r}dX_{r}^{i} $.
\mlabel{lem:inte1}
\end{lemma}

\begin{proof}
We only need to prove (\ref{eq:remainry0}) for $\Z:=\R{$\Int_0^\bullet$} \Y_rdX_{r}^{i}$, which is divided into the following two cases.

{\bf Case 1.} $\tau=\etree$. Then
\begin{align*}
R\Z_{s,t}^{\etree} =&\ \langle\etree,\R{$\Int_0^t$}\Y_{r}dX_{r}^{i}\rangle-\langle \X_{s,t}\star \etree, \R{$\Int_0^s$}\Y_{r}dX_{r}^{i}\rangle\\
=&\ \langle\etree,\R{$\Int_0^t$}\Y_{r}dX_{r}^{i}\rangle-\langle \X_{s,t}, \R{$\Int_0^s$}\Y_{r}dX_{r}^{i}\rangle\\
=&\ \langle\etree,\R{$\Int_0^t$}\Y_{r}dX_{r}^{i}\rangle-\sum_{\tau\in \mathcal{F}^{\leq (N-1)}}\langle \tau,\R{$\Int_0^s$}\Y_{r}dX_{r}^{i}\rangle\langle \X_{s,t}, \tau \rangle\\
=&\ \langle\etree,\R{$\Int_0^t$}\Y_{r}dX_{r}^{i}\rangle-\langle\etree,\R{$\Int_0^s$}\Y_{r}dX_{r}^{i}\rangle-\sum_{\tau\in \mathcal{F}^{\leq (N-1)}_+}\langle \tau,\R{$\Int_0^s$}\Y_{r}dX_{r}^{i}\rangle\langle \X_{s,t}, \tau \rangle\\
=&\ \R{$\Int_s^t$}Y_{r}dX_{r}^{i}-\sum_{[\tau]_i\in \mathcal{F}^{\leq (N-1)}} \langle \tau, \Y_{s} \rangle\langle \X_{s,t}, [\tau]_i \rangle,
\end{align*}
which implies from (\ref{eq:inte1-}) and (\ref{eq:inte2-}) that $| R\Z_{s,t}^{\etree} |=O(| t-s |^{N\alpha} ).$

{\bf Case 2.} $\tau\in \f^{\le N-1}\setminus\{\etree\}$. If $\tau = [\mu]_i$ for some $\mu\in\f^{\le N-2}$, then
\begin{align*}
R\Z_{s,t}^{[\mu]_i}
=&\ \langle [\mu]_i, \R{$\Int_0^t$}\Y_{r}dX_{r}^{i}\rangle - \langle \X_{s,t}\star[\mu]_i, \R{$\Int_0^s$}\Y_{r}dX_{r}^{i}\rangle \hspace{0.5cm} (\text{by (\ref{eq:remainry})})\\
=&\ \langle\mu,\Y_t\rangle-\sum_{\nu\in\f^{\le N-2}}\langle\nu,\Y_s\rangle\langle\X_{s,t}\star\mu,\nu\rangle\\
=&\ \langle\mu,\Y_t\rangle-\langle\X_{s,t}\star\mu,\Y_s\rangle+\sum_{\nu\in\f^{N-1}}\langle\nu,\Y_s\rangle
\langle\X_{s,t}\star\mu,\nu\rangle\\
=&\ R\Y_{s,t}^{\mu}+\sum_{\nu\in\f^{N-1}}\langle\nu,\Y_s\rangle\langle\X_{s,t}\star\mu,\nu\rangle\hspace{1cm} (\text{by (\ref{eq:remainry})})\\
=&\ O\big(|t-s|^{(N-(|\mu|+1))\alpha}\big) \hspace{1cm} (\text{by (\ref{eq:remainry0})}).
\end{align*}
Otherwise,
$$R\Z_{s,t}^{\tau}\overset{(\ref{eq:remainry})}{=}\langle \tau, \R{$\Int_0^t$}\Y_{r}dX_{r}^{i}\rangle - \langle\X_{s,t}\star \tau, \R{$\Int_0^s$}\Y_{r}dX_{r}^{i}\rangle\overset{(\ref{eq:intek1})}{=}0. $$
This completes the proof.
\end{proof}

\begin{remark}
As a componentwise version of Lemma~\mref{lem:inte1}, let
$$Z=(Z^1,\ldots,Z^d):[0,T]\rightarrow (\mathbb{R}^n)^d$$
be a path and
$$\Z=(\Z^{1}, \ldots, \Z^{d}):[0,T]\rightarrow \Big((\hmsN)^n \Big)^d,$$
where each $\Z^{i}:[0,T]\rightarrow (\hmsN)^n$ is an $\X$-controlled planarly branched rough path above $Z^i$ given in Remark~\mref{rk:gcont} for $i=1, \ldots, d$. Then we can define an $\X$-controlled planarly branched rough path
\begin{equation*}
\R{$\Int_0^\bullet$}\Z_r\cdot dX_r:[0,T]\rightarrow (\hmsN)^n,\quad t\mapsto \R{$\Int_0^t$}\Z_r \cdot dX_r:=\sum_{i=1}^{d}\R{$\Int_0^t$} \Z^i_rdX_{r}^{i}.
\end{equation*}
 Notice that each $\R{$\Int_0^t$}\Z_{r}^{i}dX_{r}^{i}$ is given in Lemma~\mref{lem:inte1}. Then
\begin{equation}
\begin{aligned}
\langle\etree,\R{$\Int_0^t$}\Z_{r}\cdot dX_{r}\rangle=&
\sum_{i=1}^{d}\langle\etree,\R{$\Int_0^t$}\Z_{r}^{i}dX_{r}^{i}\rangle\overset{(\ref{eq:intek1})}{=}\sum_{i=1}^{d}
\R{$\Int_0^t$}Z_{r}^{i}dX_{r}^{i}=:\R{$\Int_0^t$} Z_r \cdot dX_{r} \in\RR^n,\\
\langle [\tau]_{j},\R{$\Int_0^t$}\Z_{r}\cdot dX_{r}\rangle=&\sum_{i=1}^{d}\langle[\tau]_{j},\R{$\Int_0^t$}\Z_{r}^{i}dX_{r}^{i}\rangle\overset{(\ref{eq:intek1})}{=}\langle \tau,\Z_{t}^{j}\rangle\in\RR^n, \quad \forall j\in A, \tau\in \f^{\le N-2},\\
\langle \tau,\R{$\Int_0^t$}\Z_{r}\cdot dX_{r}\rangle=&\sum_{i=1}^{d}\langle\tau,\R{$\Int_0^t$}\Z_{r}^{i}dX_{r}^{i}\rangle\overset{(\ref{eq:intek1})}{=}0, \quad \forall \tau\in \f^{\le N-1}\setminus\mathcal{T}^{\le N-1}.
\end{aligned}
\mlabel{eq:inte2}
\end{equation}
\end{remark}

\subsection{Compositions of smooth functions}\label{sec:2.3}
We are going to study compositions of smooth functions with controlled planarly branched rough paths. For a map $F:\RR^n \to \RR^n$, denote~\cite{Kel}
\begin{equation}
D^mF (u):(v_1, \ldots, v_m):=\sum_{\alpha_1, \ldots, \alpha_m=1}^{n} \partial_{\alpha_1}\cdots\partial_{\alpha_m}F (u) v_1^{\alpha_1}\cdots v_m^{\alpha_m} \in \RR^n,
\mlabel{eq:ieq}
\end{equation}
where $u,v_i\in \RR^n$, $\partial_{\alpha_i}$ means to take the partial derivative with respect to the $\alpha_i$-th component of $u$, and $v^{\alpha_j}_i$ denotes the $\alpha_j$-th component of $v_i$. Here we use the convention $D^0F(u):(v_1, \ldots, v_0) := F (u)$.

The following result was shown in~\cite[Proposition 2.7]{GLM24} with roughness $\alpha\in (\frac{1}{4}, \frac{1}{3}]$.
With a similar and easier argument, we can obtain an analogous result  with roughness $\alpha\in (\frac{1}{3}, \frac{1}{2}]$.
Hence we conclude

\begin{lemma}
Let $\alpha\in (\frac{1}{4}, \frac{1}{2}]$ and $N= \lfloor \frac{1}{\alpha} \rfloor$. Let $\Y \in \cbrpx$ and $F:\RR^n \to \RR^n$ be a smooth function. Define
$$\Z: [0,T]\to (\mathcal{H}_{\mathrm{MKW}}^{\le N-1})^n,\quad t\mapsto \Z_{t} $$
by setting
\begin{equation}
\langle \etree, \Z_{t} \rangle :=\ F(Y_t), \quad \langle \tau, \Z_{t} \rangle:=\ \sum_{m=1}^{N-1} \sum_{\tau_1\cdots \tau_m=\tau} D^mF(Y_t):(\langle \tau_1, \Y_{t}  \rangle, \ldots ,\langle \tau_m, \Y_{t} \rangle), \quad \forall \tau\in \mathcal{F}^{\le (N-1)}\setminus\{\etree\}.
\mlabel{eq:zcbrp}
\end{equation}
Then $\Z$ is an $\X$-controlled planarly branched rough path above $F(Y)$, denoted by $F(\Y)$.
\mlabel{lem:compose}
\end{lemma}

In parallel,  we obtain the following two results via replacing the above $F(\Y)$ by $F(\X)$.


\begin{proposition}
Let $\alpha\in (\frac{1}{4},\frac{1}{3}]$ and $\X\in  \pbrpt $ above X. Let $F:\RR^d \to \RR^d$ be a smooth function.
Define
$$\Z: [0,T]\to (\mathcal{H}_{\mathrm{MKW}}^{\le 2})^d,\quad t\mapsto \Z_{t} $$ by setting
\begin{equation}
\begin{aligned}
\langle \etree, \Z_{t} \rangle:=&\ F(X_t) , \quad \langle \bullet_{i}, \Z_{t} \rangle:=\ \partial_iF(X_t),\\
\langle \bullet_{i}\bullet_{j}, \Z_{t} \rangle:=&\ \partial_i\partial_jF(X_t),  \quad \langle \tddeuxa{$i$}{$j$}\ \,, \Z_{t} \rangle:=\ 0, \quad \forall i,j \in A.
\end{aligned}
\mlabel{eq:zcbrp1}
\end{equation}
Then $\Z$ is an $\X$-controlled planarly branched rough path above $F(X)$, denoted by $F(\X)$.
\mlabel{pp:regu1}
\end{proposition}

\begin{proof}
By Definition~\ref{defn:cbrp}, it suffices to prove the four remainders
\begin{align}
R\Z_{s,t}^{\etree}&:=\left \langle \etree,F(\X_t)  \right \rangle- \left \langle \X_{s,t}\star \etree,F(\X_s) \right \rangle,\mlabel{eq:zcbrp2}  \\
R\Z_{s,t}^{\bullet_i}&:=\left \langle \bullet_i,F(\X_t)  \right \rangle- \left \langle \X_{s,t}\star \bullet_i,F(\X_s) \right \rangle,\mlabel{eq:zcbrp3}\\
R\Z_{s,t}^{\bullet_i\bullet_j}&:=\left \langle \bullet_i\bullet_j,F(\X_t)  \right \rangle- \left \langle \X_{s,t}\star \bullet_i\bullet_j,F(\X_s) \right \rangle, \mlabel{eq:zcbrp5}\\
R\Z_{s,t}^{\tddeuxa{$i$}{$j$}}\ \,&:=\left \langle \tddeuxa{$i$}{$j$}\ \,, F(\X_t)  \right \rangle- \left \langle \X_{s,t}\star \tddeuxa{$i$}{$j$}\ \,,F(\X_s) \right \rangle,\mlabel{eq:zcbrp4}.
\end{align}
satisfy
$$R\Z_{s,t}^{\etree}=O(|t-s|^{3\alpha}),\, \quad R\Z_{s,t}^{\bullet}=O(|t-s|^{2\alpha}),\,\quad R\Z_{s,t}^{\bullet \bullet }=O(|t-s|^{\alpha}),\quad R\Z_{s,t}^{\tddeux}=O(|t-s|^{\alpha}).$$

For the first remainder in ({\ref{eq:zcbrp2}}),
\begin{align*}
R\Z_{s,t}^{\etree}
=&\ \left \langle \etree,F(\X_t)  \right \rangle- \left \langle \X_{s,t}\star \etree,F(\X_s) \right \rangle\\
=&\ F(X_t)-\sum_{\tau\in \mathcal{F}^{\leq 2}} \langle \tau, F(\X_s)  \rangle  \langle \X_{s,t}, \tau   \rangle \hspace{2cm} (\text{by (\ref{eq:zcbrp1})})\\
=&\ F(X_t)-F(X_s)-\sum_{i=1}^{d}\langle \bullet_i, F(\X_s)  \rangle\langle \X_{s,t}, \bullet_i  \rangle-\sum_{i,j=1}^{d}\langle \bullet_i\bullet_j, F(\X_s)  \rangle  \langle \X_{s,t}, \bullet_i\bullet_j   \rangle  \\
&\  \hspace{8cm} (\text{by~(\mref{eq:leq23}) and (\ref{eq:zcbrp1})})\\
=&\ F(X_t)-F(X_s)-\sum_{i=1}^{d}\partial_iF(X_s)\langle \X_{s,t}, \bullet_i\rangle-\sum_{i,j=1}^{d}\partial_i\partial_jF(X_s)\langle \X_{s,t}, \bullet_i\bullet_j\rangle \hspace{1cm} (\text{by (\ref{eq:zcbrp1})})\\
=&\ F(X_t)-F(X_s)-\sum_{i=1}^{d}\partial_iF(X_s)\langle \X_{s,t}, \bullet_i\rangle-\frac{1}{2}\sum_{i,j=1}^{d}\partial_i\partial_jF(X_s)\langle \X_{s,t}, \bullet_i\bullet_j+\bullet_j\bullet_i\rangle\\
=&\ F(X_t)-F(X_s)-\sum_{i=1}^{d}\partial_iF(X_s)\langle \X_{s,t}, \bullet_i\rangle-\frac{1}{2}\sum_{i,j=1}^{d}\partial_i\partial_jF(X_s)\left \langle \X_{s,t}, \bullet_i \shuffle \bullet_j  \right \rangle\\
=&\ F(X_t)-F(X_s)-\sum_{i=1}^{d}\partial_iF(X_s)\langle \X_{s,t}, \bullet_i\rangle-\frac{1}{2}\sum_{i,j=1}^{d}\partial_i\partial_jF(X_s)\langle \X_{s,t}, \bullet_i\rangle\langle \X_{s,t}, \bullet_j\rangle \\
&\ \hspace{8cm} (\text{by Definition~\ref{def:pbrp}~(\ref{it:bitem2})})\\
=&\ F(X_t)-F(X_s)-\sum_{i=1}^{d}\partial_iF(X_s)(X_{t}^{i}-X_{s}^{i})-\frac{1}{2}\sum_{i,j=1}^{d}
\partial_i\partial_jF(X_s)(X_{t}^{i}-X_{s}^{i})(X_{t}^{j}-X_{s}^{j})\quad
(\text{by~(\ref{eq:above})})\\
=&\ O(|X_t-X_s|^{3}) \hspace{1cm} (\text{by the Taylor expansion of $F$ at $X_s$ of order 2})\\
=&\ O(|t-s|^{3\alpha}) \hspace{1cm} (\text{by $|X_{t}^i - X_{s}^i| = |\langle \X_{s,t} , \bullet_i \rangle|$ and Definition~\ref{def:pbrp}~(\ref{it:bitem3})}).
\end{align*}
For the remainder in (\ref{eq:zcbrp3}), similar to the above process,
\begin{align*}
R\Z_{s,t}^{\bullet_i}
=&\ \left \langle \bullet_i,F(\X_t)  \right \rangle- \left \langle \X_{s,t}\star \bullet_i,F(\X_s) \right \rangle\\
=&\ \partial_iF(X_t)-\sum_{\tau\in \mathcal{F}^{\leq 2}}\left \langle \tau, F(\X_s)  \right \rangle \left \langle \X_{s,t}\star \bullet_i, \tau  \right \rangle \hspace{2cm} (\text{by (\ref{eq:zcbrp1})})\\
=&\ \partial_iF(X_t)-\sum_{\tau\in \mathcal{F}^{\leq 2}}\left \langle \tau, F(\X_s)  \right \rangle \left \langle \X_{s,t}\otimes \bullet_i, \Delta_{\text {MKW}} (\tau)  \right \rangle \\
=&\ \partial_i F(X_t)- \partial_iF(X_s)-\sum_{j=1}^{d}\langle  \bullet_j\bullet_i, F(\X_s)  \rangle\langle \X_{s,t}, \bullet_j  \rangle \hspace{1cm} (\text{by~(\mref{eq:leq23}) and (\ref{eq:zcbrp1})})\\
=&\ \partial_iF(X_t)-\partial_iF(X_s)-\sum_{j=1}^{d}\partial_j\partial_iF(X_s)\langle \X_{s,t}, \bullet_j\rangle \hspace{1cm} (\text{by (\ref{eq:zcbrp1})})\\
=&\  \partial_iF(X_t)-\partial_iF(X_s)-\sum_{j=1}^{d}\partial_j\partial_iF(X_s)(X_{t}^{j}-X_{s}^{j})\\
=&\ O(|X_t-X_s|^{2}) \hspace{1cm} (\text{by the Taylor expansion of $F$ at $X_s$ of order 1})\\
=&\ O(|t-s|^{2\alpha}) \hspace{1cm} (\text{by $|X_{t}^i - X_{s}^i| = |\langle \X_{s,t} , \bullet_i \rangle|$ and Definition~\ref{def:pbrp}~(\ref{it:bitem3})}).
\end{align*}
For the remainder in~(\ref{eq:zcbrp5}),
\begin{align*}
R\Z_{s,t}^{\bullet_i\bullet_j}=
&\left \langle \bullet_i\bullet_j,F(\X_t)  \right \rangle- \left \langle \X_{s,t}\star \bullet_i\bullet_j,F(\X_s) \right \rangle\\
=&\ \partial_i\partial_jF(X_t)-\sum_{\tau\in \mathcal{F}^{\leq 2}}\left \langle \tau, F(\X_s)  \right \rangle \left \langle \X_{s,t}\star \bullet_i\bullet_j, \tau  \right \rangle \hspace{1cm} (\text{by (\ref{eq:zcbrp1})})\\
=&\ \partial_i\partial_jF(X_t)-\sum_{\tau\in \mathcal{F}^{\leq 2}} \langle \tau, F(\X_s)  \rangle  \langle \X_{s,t}\otimes \bullet_i\bullet_j, \Delta_{\text {MKW}}(\tau)   \rangle\\
=&\ \partial_i\partial_jF(X_t)- \langle \bullet_i\bullet_j, F(\X_s)  \rangle  \langle \X_{s,t}\otimes \bullet_i\bullet_j, \Delta_{\text {MKW}}(\bullet_i\bullet_j)   \rangle    \hspace{1cm} (\text{by~(\mref{eq:leq23}) and (\ref{eq:zcbrp1})})\\
=&\ \partial_i\partial_jF(X_t)-\partial_i\partial_jF(X_s)\hspace{1cm} (\text{by (\ref{eq:zcbrp1})})\\
=&\ O(| X_t-X_s |)\hspace{1cm} (\text{by the mean value theorem})\\
=&\ O(| t-s | ^{\alpha}) \hspace{1cm} (\text{by $|X_t-X_s|=O(| t-s | ^{\alpha})$}).
\end{align*}
Finally, for the remainder in (\ref{eq:zcbrp4}),
\begin{align*}
R\Z_{s,t}^{\tddeuxa{$i$}{$j$}}\ \,=&\
\left \langle \tddeuxa{$i$}{$j$}\ \,, F(\X_t)  \right \rangle- \left \langle \X_{s,t}\star \tddeuxa{$i$}{$j$}\ \,,F(\X_s) \right \rangle\\
=&\ -\left \langle \X_{s,t}\star \tddeuxa{$i$}{$j$}\ \,,F(\X_s) \right \rangle \hspace{1cm} (\text{by (\ref{eq:zcbrp1})})\\
=&\ -\sum_{\tau\in \mathcal{F}^{\leq 2}}\left \langle \tau, F(\X_s)  \right \rangle \left \langle \X_{s,t}\star \tddeuxa{$i$}{$j$}\ \,, \tau  \right \rangle  \\
=&\ -\sum_{\tau\in \mathcal{F}^{\leq 2}}\left \langle \tau, F(\X_s)  \right \rangle \left \langle \X_{s,t}\otimes \tddeuxa{$i$}{$j$}\ \,, \Delta_{\text {MKW}}(\tau)  \right \rangle \\
=&\ -\left \langle \tddeuxa{$i$}{$j$}\ \,, F(\X_s)  \right \rangle \left \langle \X_{s,t}\star \tddeuxa{$i$}{$j$}\ \,, \tddeuxa{$i$}{$j$}\ \, \right \rangle\\
=&\ 0 \hspace{1cm} (\text{by (\ref{eq:zcbrp1})}).
\end{align*}
This completes the proof.
\end{proof}

The following is the $N=2$ version of Proposition~\ref{pp:regu1}.

\begin{proposition}
Let $\alpha\in (\frac{1}{3},\frac{1}{2}]$ and $\X\in  \pbrps $ above X. Let $F:\RR^d \to \RR^d$ be a smooth function.
Define
$$\Z: [0,T]\to (\mathcal{H}_{\mathrm{MKW}}^{\le 1})^d,\quad t\mapsto \Z_{t} $$ by setting
\begin{equation*}
\langle \etree, \Z_t\rangle:=\ F(X_t) , \quad \langle \bullet_{i}, \Z_t\rangle:=\ \partial_iF(X_t), \quad \forall i \in A.
\end{equation*}
Then $\Z$ is an $\X$-controlled planarly branched rough path above $F(X)$, denoted by $F(\X)$.
\mlabel{pp:regu2}
\end{proposition}

\begin{proof}
The proof is similar and easier to the one of Proposition~\ref{pp:regu1}.
\end{proof}

\subsection{Rough Differential Equation}\label{sec:2.4}
Consider the RDE~(\ref{ivp}), which is understood by interpreting it as the integral equation
\begin{equation}
\left\{
\begin{array}{rll}
\delta Y_{s,t}& =\R{$\Int_s^t$}{f(Y_{r})\cdot dX_{r}} ,\quad \forall s,t\in [0, T] ,\\
Y_0 &= \xi.
\end{array}\right.
\mlabel{eq:RDE}
\end{equation}
As the case of branched rough path~\cite{Kel}, solutions to (\ref{eq:RDE}) are defined by lifting the problem to the space of $X$-controlled planarly branched rough paths.

\begin{definition}\cite{GLM24}
Let $\X:[0,T]^2\rightarrow (\mathcal{H}_{\text {MKW}}^{g})^{\le N}$ be an $\alpha$-H\"{o}lder rough path above $X$. A path $Y:[0,T]\rightarrow \mathbb{R}^n$ with $Y_0 = \xi$ is called {\bf a solution} to (\ref{eq:RDE}) if, there exists an $\X$-controlled planarly branched rough path $\Y:[0,T]\rightarrow (\hmsN)^n $ above $Y$ such that
\begin{equation}
\delta \Y_{s,t} =\R{$\Int_s^t$}{f(\Y_{r})\cdot dX_{r}} , \quad \forall s,t\in [0,T].
\mlabel{eq:BRDE}
\end{equation}
Here
$$f(\Y)=(f_1(\Y), \ldots, f_d(\Y)):[0,T]\rightarrow \Big((\hmsN)^n\Big)^d  $$ and each $f_i(\Y): [0,T]\rightarrow (\hmsN)^n$ above $f_i(Y)$ is given in (\ref{eq:zcbrp}) for $i=1, \ldots, d$.

\mlabel{defn:bRDE}
\end{definition}

Assume that the vector fields $f_i$ belong to $C_b^4(\RR^n;\RR^n)$ for $i=1,\ldots,d$. 
For $\alpha\in(\frac{1}{4},\frac{1}{3}]$ and a fixed planarly branched rough path $\X$, the existence and uniqueness of a global solution on $[0,T]$ follow from~\cite[Theorem 3.9]{GLM24}. 
The same conclusion holds for $\alpha\in(\frac{1}{3},\frac{1}{2}]$ in the corresponding setting. 
Throughout the sequel, we assume these bounded regularity conditions whenever global solutions are considered.
The solution can be represented as
\begin{equation}
\Y_t=\xi\etree+\R{$\Int_0^t$}{f(\Y_r)\cdot dX_r}, \quad \forall t\in[0,T].
\mlabel{eq:yt0}
\end{equation}
Hence for $\alpha\in (\frac{1}{4}, \frac{1}{2}]$, applying~(\mref{eq:yt0}) yields
\begin{align}
Y_t&\ =\langle \etree,\Y_t\rangle=\xi+\langle \etree, \R{$\Int_0^t$}{f(\Y_{r})\cdot dX_{r}}\rangle , \mlabel{eq:rDE3}\\
\langle [\tau _1\cdots \tau _m]_i,\Y_t\rangle&\ =\langle [\tau _1\cdots \tau _m]_i, \R{$\Int_0^t$}{f(\Y_{r})\cdot dX_{r}}\rangle ,\quad \forall\ [\tau _1\cdots \tau _m]_i\in \mathcal{T}^{\leq N-1},\mlabel{eq:rDE1}\\
\langle \tau _1\cdots \tau _m, \Y_t\rangle
&\ = \langle \tau _1\cdots \tau _m, \R{$\Int_0^t$}{f(\Y_{r})\cdot dX_{r}}\rangle \overset{(\ref{eq:intek1})}{=}0, \quad \forall \tau _1\cdots \tau _m\in \mathcal{F}^{\leq N-1}\setminus \mathcal{T}^{\leq N-1}. \mlabel{eq:RDE1}
\end{align}
Here we employ the convention that $\tau _1\cdots \tau _m:= \etree$ provided $m=0$ in~(\mref{eq:rDE1}).

As in~\cite[p.~84]{Kel}, for each $ \tau \in \mathcal{F}^{\leq N-1}$, we can write the coefficient $\langle \tau, \Y_t \rangle$
as
\begin{equation}
f_\tau(Y_t)=\langle \tau, \Y_t \rangle
\mlabel{eq:ftau}
\end{equation}
for some smooth function  $f_\tau:\RR^n \to \RR^n$.
In particular, for $\tau = \bullet_i$, we have $f_{\bullet_i}=f_i$ by~(\mref{eq:rDE1}) and~(\mref{eq:inte2}).
For $m\geq 1$, we further refine~(\ref{eq:rDE1}) as follows:
\begin{equation}
\begin{aligned}
f_{[\tau _1\cdots \tau _m]_i}(Y_t) =&\ \langle [\tau _1\cdots \tau _m]_i,\Y_t\rangle \\
=&\ \langle [\tau _1\cdots \tau _m]_i,\R{$\Int_0^t$}{f(\Y_{r})\cdot dX_{r}}\rangle \hspace{1cm}(\text{by (\ref{eq:rDE1})})\\
=&\ \langle \tau _1\cdots \tau _m, f_i(\Y_{t})\rangle \hspace{2.5cm}(\text{by (\ref{eq:inte2})})\\
=&\ \sum_{k=1}^{N-1} \sum_{\tau'_1\cdots \tau'_k=\tau_1\cdots \tau_m} D^kf_i(Y_t):(\langle \tau'_1, \Y_{t}  \rangle, \ldots ,\langle \tau'_k, \Y_{t} \rangle) \hspace{1cm}(\text{by (\ref{eq:zcbrp})})\\
=&\ D^mf_i(Y_t):(\langle \tau_1, \Y_{t}  \rangle, \ldots ,\langle \tau_m, \Y_{t} \rangle)\hspace{1cm}(\text{by (\ref{eq:RDE1})})\\
=&\ D^mf_i(Y_t):(f_{\tau _1}(Y_t), \ldots ,f_{\tau _m}(Y_t)).
\end{aligned}
\mlabel{eq:RDE2}
\end{equation}
We end this section with a kind of Taylor expansion in the framework of planarly branched rough paths.

\begin{proposition}
Let $\alpha\in (\frac{1}{4}, \frac{1}{2}]$ and $N= \lfloor \frac{1}{\alpha} \rfloor$. Let $\Y:[0,T]\rightarrow (\mathcal{H}_{\mathrm{MKW}}^{\le N-1})^n $ with $\langle \etree,\Y_s\rangle=Y_s$ be the unique \X-controlled planarly branched rough path satisfying~(\ref{eq:BRDE}). Then
\begin{equation}
\delta Y_{s,t}=\sum_{\tau\in \mathcal{T}^{\leq N}\setminus\{\etree\}}f_{\tau}(Y_s)\langle \X_{s,t}, \tau\rangle+ o(| t-s |).
\mlabel{eq:ftalor}
\end{equation}
\mlabel{pp:RDE1}
\end{proposition}

\begin{proof}
Since $\Y$ is an $\X$-controlled planarly branched rough path, we have
\begin{align*}
\delta Y_{s,t}=&\ \langle \etree,  \R{$\Int_s^t$}{f(\Y_{r})\cdot dX_{r}} \rangle \hspace{2cm}(\text{by (\ref{eq:rDE3})})\\
=&\ \sum_{i=1}^{d}\R{$\Int_s^t$}f_i(Y_{r})dX_{r}^{i} \hspace{2cm}(\text{by (\ref{eq:inte2})})\\
=&\ \sum_{i=1}^{d}f_i(Y_s)\langle\mathbf X_{s,t},\bullet_i\rangle+\sum_{i=1}^{d}\sum_{\tau\in \mathcal{F}^{\leq N-1}\setminus\{\etree\}} \langle \tau , f_i(\Y_{s}) \rangle \langle \X_{s,t} , [\tau]_i \rangle+O(| t-s |^{(N+1)\alpha})\ (\text{by (\ref{eq:inte1-}) and (\ref{eq:inte2-})})\\
=&\ \sum_{i=1}^{d}f_i(Y_s)\langle\mathbf X_{s,t},\bullet_i\rangle+\sum_{i=1}^{d}\sum_{\tau\in \mathcal{F}^{\leq N-1}\setminus\{\etree\}} \langle \tau , f_i(\Y_{s}) \rangle \langle \X_{s,t} , [\tau]_i \rangle+o(| t-s |)  \hspace{1cm}(\text{by $(N+1)\alpha>1$})\\
=&\  \sum_{i=1}^{d}f_i(Y_s)\langle\mathbf X_{s,t},\bullet_i\rangle\\
&\ +\sum_{i=1}^{d}\sum_{m=1}^{N-1}\sum_{\tau_1\cdots \tau_m=\tau \in\mathcal{F}^{\leq N-1}\setminus\{\etree\}} D^mf_i(Y_s):(\langle \tau_1, \Y_{s}  \rangle, \ldots ,\langle \tau_m, \Y_{s} \rangle)  \langle \X_{s,t} , [\tau_1\cdots\tau_m]_i \rangle +o(| t-s |) \\
&\ \hspace{11cm}(\text{by (\ref{eq:zcbrp})})\\
=&\ \sum_{i=1}^{d}f_i(Y_s)\langle\mathbf X_{s,t},\bullet_i\rangle\\
&\ +\sum_{i=1}^{d}\sum_{m=1}^{N-1}\sum_{\tau_1\cdots \tau_m=\tau\in\mathcal{F}^{\leq N-1}\setminus\{\etree\}} D^mf_i(Y_s):(f_{\tau_1}(Y_{s}) , \ldots ,f_{\tau_m}(Y_{s}))\langle \X_{s,t} , [\tau_1\cdots\tau_m]_i \rangle+o(| t-s |) \\
=&\ \sum_{i=1}^{d}f_i(Y_s)\langle\mathbf X_{s,t},\bullet_i\rangle+\sum_{i=1}^{d}\sum_{m=1}^{N-1}\sum_{\tau_1\cdots\tau_m=\tau\in \mathcal{F}^{\leq N-1}\setminus\{\etree\}}f_{[\tau _1\cdots \tau _m]_i}(Y_s)\langle \X_{s,t} , [\tau_1\cdots\tau_m]_i \rangle+ o(| t-s |) \\
&\ \hspace{11cm}(\text{by (\ref{eq:RDE2})})\\
=&\ \sum_{\tau\in \mathcal{T}^{\leq N}\setminus\{\etree\}} f_{\tau}(Y_s)\langle \X_{s,t}, \tau\rangle+o(| t-s |) \hspace{1cm}(\text{by taking $\tau:=[\tau _1\cdots \tau _m]_i$ and using $f_{\bullet_i}=f_i$}).
\end{align*}
This completes the proof.
\end{proof}

\section{ It\^o formula for planarly branched rough paths}\label{sec:3}
This section is devoted to It\^o formulas for $\alpha$-H\"{o}lder planarly branched rough paths with roughness $\alpha\in (\frac{1}{3}, \frac{1}{2}]$ and $\alpha\in (\frac{1}{4}, \frac{1}{3}]$, both of which have the simple case $F(X)$ and the general case $F(Y)$.

\subsection{It\^o formula of the simple case $F(X)$}\label{sec:3.1}
The main results in this subsection are Theorem~\ref{thm:ito1} and Theorem~\ref{thm:ito2}. Let us begin with the following lemma.

\begin{lemma}\label{lem:ex-x}
Fix $\alpha\in(1/4,1/3]$, $N=\lfloor1/\alpha\rfloor$, $A:=\{1,\ldots,d\}$ and $\widehat A:=A\cup\{(ij):i,j\in A\}$.
Denote by $\mathcal{H}_{\mathrm{MKW}}^{\le N}$ and $\widehat{\mathcal{H}}_{\mathrm{MKW}}^{\le N}$ the degree-$N$ truncations of the MKW Hopf algebras over $A$ and $\widehat A$.
Suppose that $\X$ is an $\alpha$-H\"older planarly branched rough path on $(\mathcal{H}_{\mathrm{MKW}}^{g})^{\le N}$
above $X=(X^i)_{i\in A}$. Define $\widehat X=(\widehat X^i)_{i\in\widehat A}$ by
\begin{equation}
\widehat X^i:=X^i,\quad \forall i\in A \,\text{ and }\, \widehat X_t^{(ij)}:=\langle\X_{0,t}, \bullet_j\bullet_i-\tddeuxa{$i$}{$j$}\ \ \rangle, \quad\forall  i,j\in A.
\mlabel{eq:ex-5}
\end{equation}
Then there exists an $\alpha$-H\"older planarly branched rough path $\EX$ on $(\widehat{\mathcal{H}}_{\mathrm{MKW}}^{g})^{\le N}$ above $\widehat X$ such that
\begin{enumerate}
\item  $\langle \EX_{s,t},  \tau\rangle=\langle \X_{s,t},  \tau\rangle$, \quad $\forall\tau \in \mathcal{F}^{\le N}_A$,

\item $\langle \EX_{s,t},  \bullet_{(ij)}\rangle=\langle \X_{s,t},  \bullet_j\bullet_i- \tddeuxa{$i$}{$j$}\ \, \rangle$, \quad $\forall i, j \in A$.
\end{enumerate}
Moreover, if $\tau$ is a planar rooted tree with $|\tau|\le N$ whose root is decorated by $(ij)$, then
\[
|\langle\EX_{s,t},\tau\rangle|
\le C_\tau |t-s|^{(|\tau|+1)\alpha}.
\]
\end{lemma}

\begin{proof}
Let $\hmr$ be the reduced coproduct of $\widehat{\mathcal H}_{\mathrm{MKW}}$. 
That is, $\hmr(\sigma) := \Delta_{\text {MKW}}(\sigma)-1\otimes \sigma-\sigma\otimes 1$ for $\sigma\in \mathcal F_{\widehat A}$.
Define the weight function $\omega:\widehat A\to\mathbb N$ by
\[
\omega(i)=1,\quad \forall i\in A;\quad \omega\big((ij)\big)=2,\quad \forall i,j\in A,
\]
and extend it additively to planar forests.

For $i,j\in A$, set $b_{ij}:=\bullet_j\bullet_i-\tddeuxa{$i$}{$j$}\ \,$. Since
\begin{equation}
\hmr(b_{ij})=\hmr(\bullet_j\bullet_i)-\hmr(\tddeuxa{$i$}{$j$}\ \,)=\bullet_j\otimes\bullet_i-\bullet_j\otimes\bullet_i=0,
\mlabel{eq:ex-1}
\end{equation}
$B^{ij}_{s,t}:=\langle\X_{s,t},b_{ij}\rangle$ is additive and $2\alpha$-H\"older.

For $i,j,k\in A$, define
\begin{equation}
m_{k|ij}:=\bullet_k\bullet_j\bullet_i-\tdtroisuna{$i$}{$j$}{$k$}\ \,,\quad n_{k|ij}:=\bullet_k\shuffle b_{ij}-m_{k|ij}.
\mlabel{eq:ex-4}
\end{equation}
Then
\begin{equation}
\begin{aligned}
\hmr(m_{k|ij})=&\ \hmr(\bullet_k\bullet_j\bullet_i)-\hmr(\tdtroisuna{$i$}{$j$}{$k$}\ \,)\\
=&\bullet_k\otimes \bullet_j\bullet_i+\bullet_k\bullet_j \otimes \bullet_i
-(\bullet_k\otimes \tddeuxa{$i$}{$j$}\ \,+\bullet_k\bullet_j \otimes\,\bullet_i)\\
=&\bullet_k\otimes\,(\bullet_j\bullet_i-\tddeuxa{$i$}{$j$}\ \,)\\
=&\bullet_k\otimes\,b_{ij}
\end{aligned}
\mlabel{eq:ex-2}
\end{equation}
and
\begin{align*}
\hmr(n_{k|ij})=&\ \hmr(\bullet_k\shuffle b_{ij})-\hmr(m_{k|ij})\\
=&\bullet_k\otimes\,b_{ij}+b_{ij}\otimes \bullet_k-\bullet_k\otimes\,b_{ij}\hspace{0.5cm}(\text{by~\eqref{eq:ex-1} and~\eqref{eq:ex-2}})\\
=&\ b_{ij}\otimes \bullet_k.
\end{align*}

For each $q\geq 0$, set
\[
\widehat{\mathcal H}_{\mathrm{MKW}}^{(\omega,q)}:=\operatorname{span}_{\RR}\{\tau\in\mathcal F_{\widehat A}\mid \omega(\tau)=q\}.
\]
Since the shuffle product preserves the vertices and their decorations, and each term in the coproduct of the MKW Hopf algebra splits the vertices of a forest into two parts, we have
\[
\widehat{\mathcal H}_{\mathrm{MKW}}^{(\omega,p)}\shuffle
\widehat{\mathcal H}_{\mathrm{MKW}}^{(\omega,q)}\subseteq
\widehat{\mathcal H}_{\mathrm{MKW}}^{(\omega,p+q)}
\ \text{\,and\,}\ 
\Delta_{\mathrm{MKW}}\big(\widehat{\mathcal H}_{\mathrm{MKW}}^{(\omega,q)}\big)\subseteq
\bigoplus_{p+r=q}\widehat{\mathcal H}_{\mathrm{MKW}}^{(\omega,p)}\otimes\widehat{\mathcal H}_{\mathrm{MKW}}^{(\omega,r)}.
\]
Thus $\widehat{\mathcal H}_{\mathrm{MKW}}=\bigoplus_{q\geq 0}\widehat{\mathcal H}_{\mathrm{MKW}}^{(\omega,q)}$ is a connected graded Hopf algebra with respect to the weight grading. 
Moreover, since $\widehat A$ is finite and all weights are positive, each homogeneous component is finite-dimensional.

Analogously to the truncation in Section~\ref{sec:2.1}, we define the weight-$N$ truncation by
\[
\widehat{\mathcal H}_{\mathrm{MKW}}^{[\leq N]}:=\bigoplus_{q=0}^{N}\widehat{\mathcal H}_{\mathrm{MKW}}^{(\omega,q)}
=(\RR\mathcal F_{\widehat A}^{[\leq N]},\shuffle,\etree,\Delta_{\mathrm{MKW}},\etree^*),
\]
where $\mathcal F_{\widehat A}^{[\leq N]}:=\{\tau\in\mathcal F_{\widehat A}\mid\omega(\tau)\leq N\}$. Define a map $\widetilde\X^{[N]}:[0,T]^2\rightarrow (\widehat{\mathcal H}_{\mathrm{MKW}}^{g})^{[\le N]}$ by the following identities, for every $s,t\in[0,T]$:
\begin{equation}
\begin{aligned}
&\ \langle\widetilde\X^{[N]}_{s,t},\etree\rangle:=1,\quad
\langle\widetilde\X^{[N]}_{s,t},\tau\rangle:=\langle\X_{s,t},\tau\rangle,\quad
\langle\widetilde\X^{[N]}_{s,t},\bullet_\beta\rangle:=\langle\X_{s,t},b_{ij}\rangle,\\
&\ \langle\widetilde\X^{[N]}_{s,t},\bullet_k\bullet_\beta\rangle=\langle\widetilde\X^{[N]}_{s,t},\tddeuxa{$\beta$}{$k$}\ \ \rangle
:=\langle\X_{s,t},m_{k|ij}\rangle,\ \  
\langle\widetilde\X^{[N]}_{s,t},\bullet_\beta\bullet_k\rangle=\langle\widetilde\X^{[N]}_{s,t},\tddeuxa{$k$}{$\beta$}\ \ \rangle
:=\langle\X_{s,t},n_{k|ij}\rangle,
\end{aligned}
\mlabel{eq:ex-3}
\end{equation}
where $i,j,k\in A$, $\beta:=(ij)$ and $\tau\in\mathcal F_A^{\le N}$.

To show that $\widetilde\X^{[N]}$ is an $\alpha$-H\"older planarly branched rough path, we first verify Chen's relation.
For the mixed components, the corresponding reduced coproducts in $\widehat{\mathcal H}_{\mathrm{MKW}}$ are
\[
\hmr(\bullet_k\bullet_\beta)=\hmr(\tddeuxa{$\beta$}{$k$}\ \,\,)=\bullet_k\otimes\bullet_\beta, \quad
\hmr(\bullet_\beta\bullet_k)=\hmr(\tddeuxa{$k$}{$\beta$}\ \,\,)=\bullet_\beta\otimes\bullet_k.
\]
By the identities for $\hmr(b_{ij})$, $\hmr(m_{k|ij})$ and $\hmr(n_{k|ij})$ established above, together with Chen's relation for $\X$, the mixed components defined in \eqref{eq:ex-3} satisfy Chen's relation.
Consequently,
\[
\widetilde\X^{[N]}_{s,u}\star\widetilde\X^{[N]}_{u,t}=\widetilde\X^{[N]}_{s,t}.
\]

Next, we verify the truncated character property.
Since the components involving only original decorations are inherited from $\X$, the only additional shuffle identity of total weight at most three is
\[
\bullet_k\shuffle\bullet_\beta =\bullet_k\bullet_\beta+\bullet_\beta\bullet_k.
\]
By the definitions of $m_{k|ij}$ and $n_{k|ij}$, we have
\begin{equation*}
\begin{aligned}
\langle\widetilde\X^{[N]}_{s,t},\bullet_k\shuffle\bullet_\beta\rangle&=\langle\widetilde\X^{[N]}_{s,t}, \bullet_k\bullet_\beta+\bullet_\beta\bullet_k\rangle\\
&=\langle\X_{s,t},m_{k|ij}+n_{k|ij}\rangle\hspace{1cm} (\text{by \eqref{eq:ex-3}})\\
&=\langle\X_{s,t},\bullet_k\shuffle b_{ij}\rangle\hspace{1cm} (\text{by \eqref{eq:ex-4}})\\
&=\langle\X_{s,t},\bullet_k\rangle\langle\X_{s,t},b_{ij}\rangle\hspace{1cm} (\text{by Definition~\ref{def:pbrp}~(\ref{it:bitem2})})\\
&=\langle\widetilde\X^{[N]}_{s,t},\bullet_k\rangle\langle\widetilde\X^{[N]}_{s,t},\bullet_\beta\rangle\hspace{1cm} (\text{by \eqref{eq:ex-3}}).
\end{aligned}
\end{equation*}
Thus the truncated character property holds.

Moreover, the components involving only original decorations satisfy the required weighted H\"older estimates because their weights coincide with their ordinary degrees.
The new one-vertex components are $2\alpha$-H\"older, since $b_{ij}$ is homogeneous of ordinary degree two.
Similarly, the mixed components of weight three are $3\alpha$-H\"older, since $m_{k|ij}$ and $n_{k|ij}$ are homogeneous of ordinary degree three.
Thus,
\[
|\langle\widetilde\X^{[N]}_{s,t},\tau\rangle|\le C_{\tau} |t-s|^{\alpha\omega(\tau)}, \quad 1\le\omega(\tau)\le N.
\]
Hence $\widetilde\X^{[N]}$ is an $\alpha$-H\"older planarly branched rough path.

Since $(N+1)\alpha>1$, Proposition~4.5 of \cite{BZ22} yields a unique extension to weight $2N$, denoted by $\widetilde\X^{[2N]}$.
Moreover,
\[
|\langle\widetilde\X_{s,t}^{[2N]}, \tau\rangle| \le C_{\tau}|t-s|^{\alpha\omega(\tau)}
\]
for $\tau\in\mathcal F_{\widehat A}^{[\le 2N]}\setminus\{\etree\}$. 
We now define $\EX$ by restricting $\widetilde\X^{[2N]}$ to $\widehat{\mathcal{H}}_{\mathrm{MKW}}^{\le N}$, namely,
\[
\langle\EX_{s,t},\tau\rangle:=\langle\widetilde\X_{s,t}^{[2N]},\tau\rangle, \quad |\tau|\le N.
\]
Directly, this restriction satisfies Chen's relation and the truncated character property.
Furthermore, $\omega(\tau)\ge |\tau|$ implies
\[
|\langle\EX_{s,t},\tau\rangle|\le C_{\tau}|t-s|^{\alpha|\tau|}, \quad 1\le |\tau|\le 3.
\]
Hence $\EX$ is an $\alpha$-H\"older planarly branched rough path in $(\widehat{\mathcal H}_{\mathrm{MKW}}^{g})^{\le N}$.

Consequently, $\EX$ satisfies both properties $(a)$ and $(b)$ of the lemma.
Finally, if $\tau$ is a planar rooted tree whose root is decorated by $(ij)$, then $\omega(\tau)\ge |\tau|+1$.
Therefore,
\[
|\langle\EX_{s,t},\tau\rangle| \le C_{\tau} |t-s|^{(|\tau|+1)\alpha},
\]
which completes the proof.
\end{proof}

For $\alpha\in(1/3,1/2]$, the result follows similarly by prescribing the bracket coordinates of weight two and applying Proposition~4.5 of~\cite{BZ22}, since $3\alpha>1$.

\begin{remark}
Let $\EX$ be the bracket extension of $\X$. An $\X$-controlled planarly branched rough path $\Y$ is also $\EX$-controlled. Indeed, for
$\tau\in (\mathcal{H}_{\text {MKW}}^{g})^{\le N-1}$,
\begin{align*}
\langle \tau,\Y_t   \rangle- \langle \EX_{s,t}\star \tau,\Y_s  \rangle
=&\ \langle \tau,\Y_t   \rangle- \langle \EX_{s,t}\otimes \tau,\Delta_{\text {MKW}}(\Y_s)  \rangle\\
=&\ \langle \tau,\Y_t   \rangle- \langle \X_{s,t}\otimes \tau,\Delta_{\text {MKW}}(\Y_s)  \rangle \hspace{1cm}(\text{by $\Y_s\in \mathcal{H}_{\mathrm{MKW}}$ and Lemma~\ref{lem:ex-x}})\\
=&\ R\Y^{\tau}_{s,t}=O(|t-s|^{(N-|\tau|)\alpha}) \hspace{1cm}(\text{by (\ref{eq:remainry0}) and (\ref{eq:remainry})}).
\end{align*}
If $\tau\in (\mathcal{\widehat{H}}_{\text {MKW}}^{g})^{\le N-1} \setminus (\mathcal{H}_{\text {MKW}}^{g})^{\le N-1}$, then
\begin{align*}
\langle \tau,\Y_t   \rangle- \langle \EX_{s,t}\star \tau,\Y_s  \rangle
=&\ \langle \tau,\Y_t   \rangle- \langle \EX_{s,t}\otimes \tau,\Delta_{\text {MKW}}(\Y_s)  \rangle\\
=&\ 0 \hspace{1cm}(\text{by $\Y_s , \Y_t\in \mathcal{H}_{\mathrm{MKW}}$ and $\tau\in(\mathcal{\widehat{H}}_{\text {MKW}}^{g})^{\le N-1} \setminus (\mathcal{H}_{\text {MKW}}^{g})^{\le N-1}$}),
\end{align*}
as required.
\mlabel{rem:cbrp}
\end{remark}

The following notations will be employed in Theorems~\mref{thm:ito1} and~\mref{thm:ito2}:
\begin{equation}
 \R{$\Int_s^t$}DF(Y_r):dX_r:= \sum_{i=1}^{d}\R{$\Int_s^t$} \partial_iF(Y_r)dX^{i}_{r}, \quad
\R{$\Int_s^t$}D^2F(Y_r):d\widehat{X}_r:= \sum_{i,j=1}^{d}\R{$\Int_s^t$} \partial_i\partial_jF(Y_r)d\widehat{X}^{(ij)}_{r},
\mlabel{eq:note}
\end{equation}
where all integrals on the right hand side are rough integrals.
Now we come to the It\^o formula for planarly branched rough paths with roughness $\alpha\in (\frac{1}{3},\frac{1}{2}]$.

\begin{theorem}
Let $\alpha\in (\frac{1}{3},\frac{1}{2}]$, $\X\in  \pbrps $ above $X=(X^1, \ldots, X^d)$. Let $\EX$ above $\widehat{X}$ be the bracket extension of\,
$\X$ and $F:\RR^d \to \RR^d$ a smooth function. Then
\begin{equation}
\delta F(X)_{s,t}=\R{$\Int_s^t$}DF(X_r): dX_r+\R{$\Int_s^t$}D^2F(X_r):d\widehat{X}_r.
\mlabel{eq:integral1}
\end{equation}
\mlabel{thm:ito1}
\end{theorem}
\vspace{-0.5cm}

\begin{proof}
First,
\begin{align}
F(X)_{s,t}
=&\ \sum_{i=1}^{d}\partial_iF(X_s)(X_{t}^{i}-X_{s}^{i})+\frac{1}{2}\sum_{i,j=1}^{d}\partial_i
\partial_jF(X_s)(X_{t}^{i}-X_{s}^{i})(X_{t}^{j}-X_{s}^{j})+O(|X_t-X_s|^3)\nonumber\\
&\hspace{5cm} \text{(by the second-order Taylor expansion of $F$ at $X_s$)} \nonumber \\
=&\ \sum_{i=1}^{d}\partial_iF(X_s)\langle\X_{s,t}, \bullet_i\rangle+\frac{1}{2}\sum_{i,j=1}^{d}\partial_i\partial_jF(X_s)\langle\X_{s,t}, \bullet_i\rangle\langle\X_{s,t}, \bullet_j\rangle+o(|t-s|)\nonumber\\
&\ \hspace{4cm} (\text{by \X\ above $X$, $|X_t-X_s|=O(|t-s|^\alpha)$ and $3\alpha>1$})\nonumber\\
=&\ \sum_{i=1}^{d}\partial_iF(X_s)\langle\X_{s,t}, \bullet_i\rangle+\frac{1}{2}\sum_{i,j=1}^{d}\partial_i\partial_jF(X_s)\left \langle \X_{s,t}, \bullet_i \shuffle \bullet_j  \right \rangle+o(|t-s|) \nonumber\\
&\ \hspace{6cm} (\text{by Definition~\ref{def:pbrp}~(\ref{it:bitem2})})\nonumber\\
=&\ \sum_{i=1}^{d}\partial_iF(X_s)\langle\X_{s,t}, \bullet_i\rangle+\frac{1}{2}\sum_{i,j=1}^{d}\partial_i\partial_jF(X_s)\left \langle \X_{s,t}, \bullet_i\bullet_j+\bullet_j\bullet_i  \right \rangle+o(|t-s|)\nonumber\\
=&\ \sum_{i=1}^{d}\partial_iF(X_s)\langle\X_{s,t}, \bullet_i\rangle+\sum_{i,j=1}^{d}\partial_i\partial_jF(X_s)\langle\X_{s,t}, \bullet_i\bullet_j\rangle+o(|t-s|).\mlabel{eq:ito2}
\end{align}

To deal with the first summand in the above sum, we use Proposition~\ref{pp:regu2} to define
an $\X$-controlled planarly branched rough path
$\partial_iF(\X)\in \cbrpxs$ by setting
\begin{equation}
\langle \etree, \partial_iF(\X) \rangle:=\partial_iF(X), \quad  \langle \bullet_j, \partial_iF(\X) \rangle:=\partial_i\partial_jF(X), \quad \forall j \in A.
\mlabel{eq:par}
\end{equation}
Hence by Lemma~\ref{thm:inte2}, we can define the rough integral $\R{$\Int_s^t$} \partial_iF(X_r)dX_{r}^{i}$ such that
\begin{align}
&\ \R{$\Int_s^t$} \partial_iF(X_r)dX_{r}^{i} \nonumber\\
=&\ \left \langle \etree, \partial_iF(\X_s)  \right \rangle \left \langle \X_{s,t}, \bullet_i  \right \rangle+\sum_{j=1}^{d}\left \langle \bullet_j, \partial_iF(\X_s)  \right \rangle \left \langle \X_{s,t}, \tddeuxa{$i$}{$j$}\ \,  \right \rangle+O(|t-s|^{3\alpha})\hspace{1cm} (\text{by (\ref{eq:inte2-})})\nonumber\\
=&\ \left \langle \etree, \partial_iF(\X_s)  \right \rangle \left \langle \X_{s,t}, \bullet_i  \right \rangle+\sum_{j=1}^{d}\left \langle \bullet_j, \partial_iF(\X_s)  \right \rangle \left \langle \X_{s,t}, \tddeuxa{$i$}{$j$}\ \,  \right \rangle+o(|t-s|)\hspace{1cm} (\text{by $3\alpha>1$})\nonumber\\
=&\ \partial_iF(X_s)\langle \X_{s,t}, \bullet_i\rangle+\sum_{j=1}^{d}\partial_i\partial_jF(X_s) \langle \X_{s,t}, \tddeuxa{$i$}{$j$}\ \,\rangle +o(|t-s|)\hspace{1cm} (\text{by (\ref{eq:par})}).\mlabel{eq:int2}
\end{align}
Replacing the first term to the right hand side of (\ref{eq:ito2}) by (\ref{eq:int2}),
$$F(X)_{s,t}=\sum_{i=1}^{d}\Big(\R{$\Int_s^t$} \partial_iF(X_r)dX_{r}^{i}-\sum_{j=1}^{d}\partial_i\partial_jF(X_s) \langle\X_{s,t}, \tddeuxa{$i$}{$j$}\ \,\rangle\Big)+\sum_{i,j=1}^{d}\partial_i\partial_jF(X_s)\langle \X_{s,t}, \bullet_i\bullet_j\rangle+o(|t-s|). $$
Moving the first term on the right hand side to the left yields
\begin{align}
&\ F(X)_{s,t}-\sum_{i=1}^{d}\R{$\Int_s^t$} \partial_iF(X_r)dX_{r}^{i} \nonumber \\
=&\ \sum_{i,j=1}^{d}\partial_i\partial_jF(X_s)\langle \X_{s,t}, \bullet_i\bullet_j\rangle-\sum_{i,j=1}^{d}\partial_i\partial_jF(X_s) \langle\X_{s,t}, \tddeuxa{$i$}{$j$}\ \,\rangle+o(|t-s|)\nonumber\\
=&\ \sum_{i,j=1}^{d}\partial_i\partial_jF(X_s) \langle \X_{s,t}, \bullet_j\bullet_i\rangle-\sum_{i,j=1}^{d}\partial_i\partial_jF(X_s) \langle\X_{s,t}, \tddeuxa{$i$}{$j$}\ \,\rangle+o(|t-s|)\nonumber\\
&\ \hspace{2cm} (\text{by exchanging $i$ and $j$ in the first summand and $\partial_i\partial_j=\partial_j\partial_i $}) \nonumber\\
=&\ \sum_{i,j=1}^{d}\partial_i\partial_jF(X_s) \langle \X_{s,t}, \bullet_j\bullet_i- \tddeuxa{$i$}{$j$}\ \,  \rangle+o(|t-s|).\mlabel{eq:int3}
\end{align}
Now we focus on the first term on the right hand side.
Again using Proposition~\ref{pp:regu2}, we can define an $\X$-controlled planarly branched rough path $\partial_i\partial_j F(\X)$, which
is also $\EX$-controlled by Remark~\ref{rem:cbrp}.
So we can use Lemma~\ref{thm:inte2} to define the rough integral
$\R{$\Int_s^t$} \partial_i\partial_jF(X_r)d\widehat{X}_{r}^{(ij)},$
which satisfies
\begin{align}
\R{$\Int_s^t$} \partial_i\partial_jF(X_r)d\widehat{X}_{r}^{(ij)}
=&\ \langle \etree, \partial_i\partial_jF(\X_s)   \rangle  \langle \EX_{s,t}, \bullet_{(ij)}   \rangle+\sum_{k=1}^d\langle \bullet_k, \partial_i\partial_jF(\X_s)  \rangle \langle \EX_{s,t}, \tddeuxa{$(ij)$}{$k$}\ \ \ \ \rangle+O(|t-s|^{3\alpha}) \nonumber\\
&\ \hspace{8cm} (\text{by Lemma~\ref{thm:inte2}})\nonumber\\
=&\ \langle \etree, \partial_i\partial_jF(\X_s)   \rangle  \langle \EX_{s,t}, \bullet_{(ij)}   \rangle+o(|t-s|)\nonumber\\
&\ \hspace{5cm} (\text{by $|\langle \EX_{s,t}, \tddeuxa{$(ij)$}{$k$}\ \ \ \ \rangle|=O(|t-s|^{3\alpha})$ and $3\alpha>1$})\nonumber\\
=&\ \partial_i\partial_jF(X_s)\langle \X_{s,t},\bullet_j\bullet_i- \tddeuxa{$i$}{$j$}\ \, \rangle+o(|t-s|) \hspace{1.5cm} (\text{by Lemma~\ref{lem:ex-x}}).\mlabel{eq:int4}
\end{align}
Using (\ref{eq:int4}) to replace $\partial_i\partial_jF(X_s)\langle \X_{s,t},\bullet_j\bullet_i- \tddeuxa{$i$}{$j$}\ \, \rangle$ in (\ref{eq:int3}),
\begin{equation}
F(X)_{s,t}-\sum_{i=1}^{d}\R{$\Int_s^t$} \partial_iF(X_r)dX_{r}^{i}-\sum_{i,j=1}^{d}\R{$\Int_s^t$} \partial_i\partial_jF(X_r)d\widehat{X}_{r}^{(ij)}=o(|t-s|).
\mlabel{eq:ito5}
\end{equation}
Since the left hand side is an increment, it must be zero.
Hence
\begin{align*}
F(X)_{s,t}
=&\ \sum_{i=1}^{d}\R{$\Int_s^t$} \partial_iF(X_r)dX_{r}^{i}+\sum_{i,j=1}^{d}\R{$\Int_s^t$} \partial_i\partial_jF(X_r)d\widehat{X}_{r}^{(ij)}\\
=&\ \R{$\Int_s^t$} DF(X_r):dX_{r}+\R{$\Int_s^t$}D^2F(X_r):d\widehat{X}_{r}  \hspace{1cm} (\text{by (\ref{eq:note})}),
\end{align*}
as needed.
\end{proof}

Next, we turn to the case of $\alpha\in (\frac{1}{4},\frac{1}{3}]$. 
Using the coproduct of the MKW Hopf algebra, we obtain
\begin{equation}
\begin{aligned}
&\ \hmr(\bullet_k\bullet_j\bullet_i-\tdtroisuna{$i$}{$j$}{$k$}\ \,-\tddeuxa{$(ij)$}{$k$}\ \ \ \ )\\
=&\ \Big(\bullet_k\otimes\bullet_j\bullet_i+\bullet_k\bullet_j\otimes\bullet_i\Big)-\Big(\bullet_k\otimes\tddeuxa{$i$}{$j$}\ \,+\bullet_k\bullet_j\otimes\bullet_i \Big)-\bullet_k\otimes\bullet_{(ij)}\\
=&\ \bullet_k\otimes\,\Big(\bullet_j\bullet_i-\tddeuxa{$i$}{$j$}\ \,- \bullet_{(ij)} \Big).
\end{aligned}
\mlabel{eq:+1}
\end{equation}
Although $\bullet_k\bullet_j\bullet_i-\tdtroisuna{$i$}{$j$}{$k$}\ \,-\tddeuxa{$(ij)$}{$k$}\ \ \ \ $ is not primitive in the MKW Hopf algebra, Lemma~\ref{lem:ex-x} ensures that
\[
\langle\EX_{s,t}, \bullet_j\bullet_i-\tddeuxa{$i$}{$j$}\ \,- \bullet_{(ij)} \rangle=0.
\]
Consequently, for $0\le s\le u\le t\le T$, Chen's relation yields
\begin{equation}
\begin{aligned}
\langle\EX_{s,t}, \bullet_k\bullet_j\bullet_i-\tdtroisuna{$i$}{$j$}{$k$}\ \,-\tddeuxa{$(ij)$}{$k$}\ \ \ \ \rangle
=&\ \langle\EX_{s,u}\star\EX_{u,t}, \bullet_k\bullet_j\bullet_i-\tdtroisuna{$i$}{$j$}{$k$}\ \,-\tddeuxa{$(ij)$}{$k$}\ \ \ \ \rangle\\
=&\ \langle\EX_{s,u}\otimes \EX_{u,t}, \Delta_{ \text {MKW}}(\bullet_k\bullet_j\bullet_i-\tdtroisuna{$i$}{$j$}{$k$}\ \,-\tddeuxa{$(ij)$}{$k$}\ \ \ \ )\rangle\\
=&\ \langle\EX_{s,u}, \bullet_k\bullet_j\bullet_i-\tdtroisuna{$i$}{$j$}{$k$}\ \,-\tddeuxa{$(ij)$}{$k$}\ \ \ \ \rangle +\langle\EX_{u,t},\bullet_k\bullet_j\bullet_i-\tdtroisuna{$i$}{$j$}{$k$}\ \,-\tddeuxa{$(ij)$}{$k$}\ \ \ \ \rangle\\
&+\langle\EX_{s,u},\bullet_k\rangle \langle\EX_{u,t}, \bullet_j\bullet_i-\tddeuxa{$i$}{$j$}\ \,- \bullet_{(ij)} \rangle\hspace{1cm} (\text{by (\ref{eq:+1}}))\\
=&\ \langle\EX_{s,u}, \bullet_k\bullet_j\bullet_i-\tdtroisuna{$i$}{$j$}{$k$}\ \,-\tddeuxa{$(ij)$}{$k$}\ \ \ \ \rangle +\langle\EX_{u,t},\bullet_k\bullet_j\bullet_i-\tdtroisuna{$i$}{$j$}{$k$}\ \,-\tddeuxa{$(ij)$}{$k$}\ \ \ \ \rangle.
\end{aligned}
\end{equation}
Thus, the increments $\langle\EX_{s,t}, \bullet_k\bullet_j\bullet_i-\tdtroisuna{$i$}{$j$}{$k$}\ \,-\tddeuxa{$(ij)$}{$k$}\ \ \ \ \rangle$ are additive.

\vspace{0.2cm}
For $i,j,k=1,\ldots,d$, define $\widetilde{X}^{(ijk)}:[0,T] \to \RR$ by setting 
\begin{equation*}
\widetilde{X}^{(ijk)}_{t}:=\langle \EX_{0,t},\bullet_k\bullet_j\bullet_i-\tdtroisuna{$i$}{$j$}{$k$}\ \,-\tddeuxa{$(ij)$}{$k$}\ \ \ \ \rangle,\quad \forall t\in [0, T].
\end{equation*}
The additivity established above implies
\begin{equation}
\delta\widetilde{X}^{(ijk)}_{s, t}=\langle \EX_{s,t},\bullet_k\bullet_j\bullet_i-\tdtroisuna{$i$}{$j$}{$k$}\ \,-\tddeuxa{$(ij)$}{$k$}\ \ \ \ \rangle.
\mlabel{eq:xhat}
\end{equation}

\begin{theorem}
Let $\alpha\in (\frac{1}{4},\frac{1}{3}]$, $\X\in  \pbrpt $ above $X=(X^1, \ldots, X^d)$. Let $\EX$ above $\widehat{X}$ in~(\ref{eq:ex-5}) be the bracket extension of
\,$\X$ and $F:\RR^d \to \RR^d$ a smooth function. Then
\begin{equation}
\delta F(X)_{s,t}=\R{$\Int_s^t$}DF(X_r): dX_r+\R{$\Int_s^t$}D^2F(X_r):d\widehat{X}_r+\R{$\Int_s^t$}D^3F(X_r):d\widetilde{X}_r,
\mlabel{eq:integral2}
\end{equation}
where the first and second integrals are from~(\ref{eq:note}), and the third integral is understood as a Young integral~\cite{You}.
\mlabel{thm:ito2}
\end{theorem}

\begin{proof}
We have
\begin{align}
F(X)_{s,t}
=&\ \sum_{i=1}^{d}\partial_iF(X_s)(X_{t}^{i}-X_{s}^{i})+\frac{1}{2}\sum_{i,j=1}^{d}\partial_i\partial_jF(X_s)(X_{t}^{i}-X_{s}^{i})(X_{t}^{j}-X_{s}^{j})\nonumber\\
&\ +\frac{1}{6}\sum_{i,j,k=1}^{d}\partial_i\partial_j\partial_kF(X_s)(X_{t}^{i}-X_{s}^{i})(X_{t}^{j}-X_{s}^{j})(X_{t}^{k}-X_{s}^{k})+O(|X_t-X_s|^4)\nonumber\\
%
&\ \hspace{4cm} \text{(by the third-order Taylor expansion of $F$ at $X_s$)} \nonumber \\
=&\ \sum_{i=1}^{d}\partial_iF(X_s)\langle\X_{s,t}, \bullet_i\rangle+\frac{1}{2}\sum_{i,j=1}^{d}\partial_i\partial_jF(X_s)\langle\X_{s,t}, \bullet_i\rangle \langle\X_{s,t}, \bullet_j\rangle\nonumber\\
&\ +\frac{1}{6}\sum_{i,j,k=1}^{d}\partial_i\partial_j\partial_kF(X_s)\langle\X_{s,t}, \bullet_i\rangle\langle\X_{s,t}, \bullet_j\rangle\langle\X_{s,t}, \bullet_k\rangle+o(|t-s|)\nonumber\\
&\ \hspace{4cm} (\text{by \X\ above $X$, $|X_t-X_s|=O(|t-s|^\alpha)$ and $4\alpha>1$})\nonumber\\
=&\ \sum_{i=1}^{d}\partial_iF(X_s)\langle\X_{s,t}, \bullet_i\rangle+\frac{1}{2}\sum_{i,j=1}^{d}\partial_i\partial_jF(X_s)\left \langle \X_{s,t}, \bullet_i \shuffle \bullet_j  \right \rangle\nonumber\\
&\ +\frac{1}{6}\sum_{i,j,k=1}^{d}\partial_i\partial_j\partial_kF(X_s)\left \langle \X_{s,t}, (\bullet_i \shuffle \bullet_j)\shuffle \bullet_k  \right \rangle+o(|t-s|) \nonumber\\
&\ \hspace{3cm} (\text{the second and third summand employ Definition~\ref{def:pbrp}~(\ref{it:bitem2})})\nonumber\\
=&\ \sum_{i=1}^{d}\partial_iF(X_s)\langle\X_{s,t}, \bullet_i\rangle+\frac{1}{2}\sum_{i,j=1}^{d}\partial_i\partial_jF(X_s)\left \langle \X_{s,t}, \bullet_i\bullet_j+\bullet_j\bullet_i  \right \rangle+\frac{1}{6}\sum_{i,j,k=1}^{d}\partial_i\partial_j\partial_kF(X_s)\nonumber\\
&\ \left \langle \X_{s,t}, \bullet_i\bullet_j\bullet_k+\bullet_i\bullet_k\bullet_j+\bullet_k\bullet_i\bullet_j+\bullet_j\bullet_i\bullet_k+\bullet_j\bullet_k\bullet_i+\bullet_k\bullet_j\bullet_i  \right \rangle+o(|t-s|)\nonumber\\
=&\ \sum_{i=1}^{d}\partial_iF(X_s)\langle\X_{s,t}, \bullet_i\rangle+\sum_{i,j=1}^{d}\partial_i\partial_jF(X_s)\langle\X_{s,t}, \bullet_i\bullet_j\rangle+\sum_{i,j,k=1}^{d}\partial_i\partial_j\partial_kF(X_s)\langle\X_{s,t}, \bullet_k\bullet_j\bullet_i\rangle\nonumber\\
&\ +o(|t-s|).\mlabel{eq:ito4}
\end{align}

For the first term on the right hand side, with a similar argument to~(\ref{eq:int2}),
\begin{align}
&\ \R{$\Int_s^t$} \partial_iF(X_r)dX_{r}^{i}\nonumber\\
=&\ \langle \etree, \partial_iF(\X_s)  \rangle \langle \X_{s,t}, \bullet_i  \rangle+\sum_{j=1}^{d}\left \langle \bullet_j, \partial_iF(\X_s)  \right \rangle \left \langle \X_{s,t}, \tddeuxa{$i$}{$j$}\ \,  \right \rangle+\sum_{j,k=1}^{d} \langle \bullet_k\bullet_j, \partial_iF(\X_s)   \rangle  \langle \X_{s,t}, \tdtroisuna{$i$}{$j$}{$k$}\ \,\rangle\nonumber\\
&\ +\sum_{j,k=1}^{d} \langle \tddeuxa{$j$}{$k$}\ \,, \partial_iF(\X_s)   \rangle  \langle \X_{s,t}, \tdddeuxa{$i$}{$j$}{$k$}\ \,\rangle+O(|t-s|^{4\alpha})\hspace{1cm} (\text{by (\ref{eq:inte1-}}))\nonumber\\
=&\ \langle \etree, \partial_iF(\X_s)  \rangle \langle \X_{s,t}, \bullet_i  \rangle+\sum_{j=1}^{d}\left \langle \bullet_j, \partial_iF(\X_s)  \right \rangle \left \langle \X_{s,t}, \tddeuxa{$i$}{$j$}\ \,  \right \rangle+\sum_{j,k=1}^{d} \langle \bullet_k\bullet_j, \partial_iF(\X_s)   \rangle  \langle \X_{s,t}, \tdtroisuna{$i$}{$j$}{$k$}\ \,\rangle\nonumber\\
&\ +\sum_{j,k=1}^{d} \langle \tddeuxa{$j$}{$k$}\ \,, \partial_iF(\X_s)   \rangle  \langle \X_{s,t}, \tdddeuxa{$i$}{$j$}{$k$}\ \,\rangle+o(|t-s|)\hspace{1cm} (\text{by $4\alpha>1$})\nonumber\\
=&\ \partial_iF(X_s)\langle \X_{s,t}, \bullet_i   \rangle+\sum_{j=1}^{d}\partial_i\partial_jF(X_s) \langle\X_{s,t}, \tddeuxa{$i$}{$j$}\ \,\rangle +\sum_{j,k=1}^{d}\partial_i\partial_j\partial_kF(X_s)  \langle\X_{s,t}, \tdtroisuna{$i$}{$j$}{$k$}\ \,\rangle+o(|t-s|)\mlabel{eq:intb2}\\
&\ \hspace{7cm} (\text{by Proposition~\ref{pp:regu1}}).\nonumber
\end{align}
Using (\ref{eq:intb2}) in place of the first term to the right hand side of (\ref{eq:ito4}),
\begin{align}
&\ F(X)_{s,t}-\sum_{i=1}^{d}\R{$\Int_s^t$} \partial_iF(X_r)dX_{r}^{i} \nonumber\\
=&\ \sum_{i,j=1}^{d}\partial_i\partial_jF(X_s) \langle \X_{s,t},\bullet_j\bullet_i- \tddeuxa{$i$}{$j$}\ \, \rangle+\sum_{i,j,k=1}^{d}\partial_i\partial_j\partial_kF(X_s) \langle \X_{s,t},\bullet_k\bullet_j\bullet_i-\tdtroisuna{$i$}{$j$}{$k$}\ \, \rangle+o(|t-s|).\mlabel{eq:intb3}
\end{align}
We are going to deal with the first term on the right hand side. Similar to (\ref{eq:int4}),
\begin{align}
&\ \R{$\Int_s^t$} \partial_i\partial_jF(X_r)d\widehat{X}_{r}^{(ij)}\nonumber\\
=&\ \langle \etree, \partial_i\partial_jF(\X_s)   \rangle \langle \EX_{s,t}, \bullet_{(ij)}  \rangle+\sum_{k=1}^{d}\left \langle \bullet_k, \partial_i\partial_jF(\X_s)  \right \rangle \left \langle \EX_{s,t}, \tddeuxa{$(ij)$}{$k$}\ \ \ \ \right \rangle+\sum_{k,m=1}^{d}\langle \bullet_k\bullet_m, \partial_i\partial_jF(\X_s)   \rangle \langle \EX_{s,t}, \tdtroisuna{$(ij)$}{$m$}{$k$}\ \ \ \rangle\nonumber\\
&\ +\sum_{k,m=1}^{d}\langle \tddeuxa{$k$}{$m$}\ \,, \partial_i\partial_jF(\X_s)   \rangle \langle \EX_{s,t}, \tdddeuxa{$(ij)$}{$k$}{$m$}\ \ \ \  \rangle+O(|t-s|^{4\alpha})\nonumber\\
=&\ \left \langle \etree, \partial_i\partial_jF(\X_s)  \right \rangle \left \langle \EX_{s,t}, \bullet_{(ij)}  \right \rangle+\sum_{k=1}^{d}\left \langle \bullet_k, \partial_i\partial_jF(\X_s)  \right \rangle \left \langle \EX_{s,t}, \tddeuxa{$(ij)$}{$k$}\ \ \ \ \right \rangle+o(|t-s|)\nonumber\\
&\ \hspace{3cm} (\text{by $|\langle \EX_{s,t}, \tdtroisuna{$(ij)$}{$m$}{$k$}\ \ \ \rangle|=O(|t-s|^{4\alpha})$, $|\langle \EX_{s,t}, \tdddeuxa{$(ij)$}{$k$}{$m$}\ \ \ \  \rangle|=O(|t-s|^{4\alpha})$ and $4\alpha>1$})\nonumber\\
=&\ \partial_i\partial_jF(X_s)\langle \X_{s,t},\bullet_j\bullet_i- \tddeuxa{$i$}{$j$}\ \, \rangle+\sum_{k=1}^{d}\partial_i\partial_j\partial_kF(X_s) \left \langle \EX_{s,t}, \tddeuxa{$(ij)$}{$k$}\ \ \ \ \right \rangle+o(|t-s|)\mlabel{eq:intb4}\\
&\ \hspace{8cm} (\text{by (\ref{eq:zcbrp1}) and Lemma~\ref{lem:ex-x}}).\nonumber
\end{align}
Applying~(\ref{eq:intb4}) to replace the $\partial_i\partial_jF(X_s)\langle \X_{s,t},\bullet_j\bullet_i- \tddeuxa{$i$}{$j$}\ \, \rangle$ in (\ref{eq:intb3}) yields
\begin{align}
&\ F(X)_{s,t}-\sum_{i=1}^{d}\R{$\Int_s^t$} \partial_iF(X_r)dX_{r}^{i}-\sum_{i,j=1}^{d}\R{$\Int_s^t$} \partial_i\partial_jF(X_r)d\widehat{X}_{r}^{(ij)}\nonumber\\
=&\ \sum_{i,j,k=1}^{d}\partial_i\partial_j\partial_kF(X_s) \langle \EX_{s,t},\bullet_k\bullet_j\bullet_i-\tdtroisuna{$i$}{$j$}{$k$}\ \,-\tddeuxa{$(ij)$}{$k$}\ \ \ \ \rangle+o(|t-s|)\nonumber\\
=&\ \sum_{i,j,k=1}^{d}\partial_i\partial_j\partial_kF(X_s)\delta \widetilde{X}^{(ijk)}_{s,t}+o(|t-s|)\hspace{1cm} (\text{by (\ref{eq:xhat}})).\mlabel{eq:intb5}
\end{align}
By Definition~\ref{def:pbrp}~(\ref{it:bitem3}) and Lemma~\ref{lem:ex-x},
$$|\delta\widetilde{X}^{(ijk)}_{s,t}|=O(| t-s | ^{3\alpha  } ).$$
Since $|X_t-X_s|=O(| t-s | ^{\alpha  } )$ and $F$ is smooth, we obtain
$$|\partial_i\partial_j\partial_kF(X_t)-\partial_i\partial_j\partial_kF(X_s)|=O(| t-s | ^{\alpha  } ).$$
Due to $\alpha+3\alpha> 1$, we can define the Young integral~\cite{You} of the first term on the right hand side of (\ref{eq:intb5})
\begin{align}
\R{$\Int_s^t$} \partial_i\partial_j\partial_kF(X_r)d\widetilde{X}^{(ijk)}_{r}
:=&\ \lim_{|\pi | \to 0} \sum_{[t_u, t_{u+1}]\in \pi}\partial_i\partial_j\partial_kF(X_{t_u})\delta\widetilde{X}^{(ijk)}_{t_u,t_{u+1}}\nonumber\\
=&\ \partial_i\partial_j\partial_kF(X_s)\delta\widetilde{X}^{(ijk)}_{s,t}+O(|t-s|^{4\alpha}) \quad \text{(by Sewing lemma~~\mcite{BZ22})}\nonumber\\
=&\  \partial_i\partial_j\partial_kF(X_s)\delta\widetilde{X}^{(ijk)}_{s,t}+o(|t-s|) \hspace{1cm}(\text{by $4\alpha>1$}).\mlabel{eq:intb16}
\end{align}
Combining (\ref{eq:intb5}) and (\ref{eq:intb16}),
$$F(X)_{s,t} - \sum_{i=1}^{d}\R{$\Int_s^t$} \partial_iF(X_r)dX_{r}^{i}-\sum_{i,j=1}^{d}\R{$\Int_s^t$} \partial_i\partial_jF(X_r)d\widehat{X}_{r}^{(ij)} - \sum_{i,j,k=1}^{d}\R{$\Int_s^t$} \partial_i\partial_j\partial_kF(X_r)d\widetilde{X}^{(ijk)}_{r} = o(|t-s|),$$
which implies that the left hand side must be zero by a similar argument to (\ref{eq:ito5}).
This completes the proof.
\end{proof}

\subsection{It\^o formula of the general case $F(Y)$}\label{sec:3.2}
The goal of this subsection is the It\^o formula of the general case with roughness $\alpha\in (\frac{1}{3}, \frac{1}{2}]$ and $\alpha\in (\frac{1}{4}, \frac{1}{3}]$, respectively.

For notational convenience, we use the shorthand
\begin{equation}
\begin{aligned}
 \R{$\Int_s^t$}DF(Y_r):\Big(f(Y_r) \cdot dX_r\Big):=&\ \sum_{i=1}^{d}\R{$\Int_s^t$} \Big(DF(Y_r):\big(f_i(Y_r)\big)\Big)dX^i_{r},\\
\R{$\Int_s^t$}D^2F(Y_r):\Big(\big(f(Y_r), f(Y_r) \big)\cdot d\widehat{X}_r\Big):=&\ \sum_{i,j=1}^{d}\R{$\Int_s^t$}\Big(D^2F(Y_r):\big(f_i(Y_r), f_j(Y_r)\big)\Big)d\widehat{X}^{(ij)}_{r},
\end{aligned}
\mlabel{eq:note2}
\end{equation}
where all integrals on the right hand side are rough integrals.

\begin{theorem}
Let $\alpha\in (\frac{1}{3},\frac{1}{2}]$, $\X\in  \pbrps $ above $X=(X^1, \ldots, X^d)$ and $\Y \in \cbrpxs$ above $Y$ be a solution to (\ref{eq:BRDE}).
Let $\widehat{\X}$ above $\widehat{X}$ in~(\ref{eq:ex-5}) be the bracket extension of\, $\X$ and $F:\RR^n \to \RR^n$ a smooth function. Then
\begin{equation}
\delta F(Y)_{s,t}=\R{$\Int_s^t$}DF(Y_r):\Big(\big(f(Y_r) \big)\cdot dX_{r}\Big)+\R{$\Int_s^t$} D^2F(Y_r):\Big(\big(f(Y_r), f(Y_r)\big)\cdot d\widehat{X}_{r}\Big).
\mlabel{eq:integral3}
\end{equation}
\mlabel{thm:yitoa}
\end{theorem}
\vspace{-0.5cm}

\begin{proof}
First,
\begin{align}
\delta Y_{s,t}
=&\ \sum_{i=1}^{d}f_{\bullet_i}(Y_s)\langle \X_{s,t}, \bullet_i\rangle+\sum_{i,j=1}^{d}f_{\tddeuxaa{$i$}{$j$}\ \,}(Y_s)\langle \X_{s,t}, \tddeuxa{$i$}{$j$}\ \,\rangle+o(|t-s|) \hspace{0.5cm}(\text{by~(\ref{eq:ftalor}) and~(\ref{eq:leq23})})\nonumber \\
=&\ \sum_{i=1}^{d}f_{i}(Y_s)\langle \X_{s,t}, \bullet_i\rangle+\sum_{i,j=1}^{d}Df_i(Y_s):\big(f_j(Y_s)\big)\langle \X_{s,t}, \tddeuxa{$i$}{$j$}\ \,\rangle+o(|t-s|)\mlabel{eq:yito1}\\
&\ \hspace{8cm}(\text{by~(\ref{eq:RDE2}) and Proposition~\ref{pp:RDE1}})\nonumber.
\end{align}
Using the second-order Taylor expansion of $F$ at $Y_s$,
\begin{align}
\delta F(Y)_{s,t}
=&\ \sum_{i=1}^{n}\partial_iF(Y_s)(Y_{t}^{i}-Y_{s}^{i})+\frac{1}{2}\sum_{i,j=1}^{n}\partial_i\partial_jF(Y_s)(Y_{t}^{i}-Y_{s}^{i})(Y_{t}^{j}-Y_{s}^{j})+O(|Y_t-Y_s|^3)\nonumber\\
=&\ \sum_{i=1}^{n}\partial_iF(Y_s)(Y_{t}^{i}-Y_{s}^{i})+\frac{1}{2}\sum_{i,j=1}^{n}\partial_i\partial_jF(Y_s)(Y_{t}^{i}-Y_{s}^{i})(Y_{t}^{j}-Y_{s}^{j})+o(|t-s|)\nonumber\\
&\ \hspace{6cm} (\text{by $|Y_t-Y_s|=O(|t-s|^\alpha)$ and $3\alpha>1$})\nonumber\\
=&\ DF(Y_s):\delta Y_{s,t}+\frac{1}{2}D^2F(Y_s):(\delta Y_{s,t}, \delta Y_{s,t})+o(|t-s|).\mlabel{eq:yito2}
\end{align}
Substituting (\ref{eq:yito1}) into (\ref{eq:yito2}),
\begin{align}
\delta F(Y)_{s,t}
=&\ \sum_{i=1}^{d}DF(Y_s):\big(f_{i}(Y_s)\langle \X_{s,t}, \bullet_i\rangle\big)+\sum_{i,j=1}^{d}DF(Y_s):\Big(Df_i(Y_s):\big(f_j(Y_s)\big)\langle \X_{s,t}, \tddeuxa{$i$}{$j$}\ \,\rangle\Big)  \nonumber \\
&\ +\frac{1}{2}\sum_{i, j=1}^{d}D^2F(Y_s):\Big(f_{i}(Y_s)\langle \X_{s,t}, \bullet_i\rangle, f_{j}(Y_s)\langle \X_{s,t}, \bullet_j\rangle\Big)+o(|t-s|)\nonumber\\
=&\ \sum_{i=1}^{d}DF(Y_s):\big(f_{i}(Y_s)\big)\langle \X_{s,t}, \bullet_i\rangle+\sum_{i,j=1}^{d}DF(Y_s):\Big(Df_i(Y_s):\big(f_j(Y_s)\big)\Big)\langle \X_{s,t}, \tddeuxa{$i$}{$j$}\ \,\rangle \nonumber\\
&\ +\frac{1}{2}\sum_{i, j=1}^{d}D^2F(Y_s):\Big(f_{i}(Y_s), f_{j}(Y_s)\Big)\langle \X_{s,t}, \bullet_i\rangle\langle \X_{s,t}, \bullet_j\rangle+o(|t-s|) \nonumber\\
&\ \hspace{7cm}(\text{by $\langle \X_{s,t}, \bullet_i\rangle, \langle \X_{s,t}, \tddeuxa{$i$}{$j$}\ \,\rangle \in \RR$})\nonumber\\
=&\  \sum_{i=1}^{d}DF(Y_s):(f_{i}(Y_s))\langle \X_{s,t}, \bullet_i\rangle+\sum_{i,j=1}^{d}DF(Y_s):\Big(Df_i(Y_s):(f_j(Y_s))\Big)\langle \X_{s,t}, \tddeuxa{$i$}{$j$}\ \,\rangle \nonumber\\
&\ +\frac{1}{2}\sum_{i, j=1}^{d}D^2F(Y_s):(f_{i}(Y_s), f_{j}(Y_s))\langle \X_{s,t}, \bullet_i \shuffle \bullet_j\rangle+o(|t-s|)\nonumber\\
&\ \hspace{3cm} (\text{the third summand uses Definition~\ref{def:pbrp}~(\ref{it:bitem2})})\nonumber\\
=&\ \sum_{i=1}^{d}DF(Y_s):\big(f_{i}(Y_s)\big)\langle \X_{s,t}, \bullet_i\rangle+\sum_{i,j=1}^{d}DF(Y_s):\Big(Df_i(Y_s):(f_j(Y_s))\Big)\langle \X_{s,t}, \tddeuxa{$i$}{$j$}\ \,\rangle \nonumber\\
&\ +\sum_{i, j=1}^{d}D^2F(Y_s):\Big(f_{i}(Y_s), f_{j}(Y_s)\Big)\langle \X_{s,t}, \bullet_j\bullet_i\rangle+o(|t-s|).\mlabel{eq:yito3}
\end{align}
For the first summand, using Lemma~\ref{lem:compose}, we can define an $\X$-controlled planarly branched rough path
$$DF(\Y_s):(f_{i}(\Y_s))\in \cbrpxs.$$
In terms of Theorem~\mref{thm:inte2}, we can define the rough integral of $DF(\Y_s):(f_{i}(\Y_s))$, which satisfies
\begin{align}
&\ \R{$\Int_s^t$}DF(Y_r):\big(f_{i}(Y_r)\big)dX_{r}^{i}\nonumber\\
=&\  \langle \etree, DF(\Y_s):f_{i}(\Y_s)\rangle \langle \X_{s,t}, \bullet_i \rangle+\sum_{j=1}^{d} \langle \bullet_j, DF(\Y_s):f_{i}(\Y_s) \rangle  \langle \X_{s,t}, \tddeuxa{$i$}{$j$}\ \,   \rangle+O(|t-s|^{3\alpha})\hspace{0.5cm} (\text{by (\ref{eq:inte2-})})\nonumber\\
=&\ DF(Y_s):f_{i}(Y_s) \langle \X_{s,t}, \bullet_i \rangle+\sum_{j=1}^{d}D(DF(Y_s):f_{i}(Y_s)):\langle \bullet_j,\Y_s \rangle \langle \X_{s,t}, \tddeuxa{$i$}{$j$}\ \,   \rangle+o(|t-s|)\nonumber\\
&\ \hspace{9cm} (\text{by (\ref{eq:zcbrp}) and $3\alpha>1$})\nonumber\\
=&\ DF(Y_s):f_{i}(Y_s) \langle \X_{s,t}, \bullet_i \rangle+\sum_{j=1}^{d}D(DF(Y_s):f_{i}(Y_s)):\big(f_{j}(Y_s) \big)\langle \X_{s,t}, \tddeuxa{$i$}{$j$}\ \,   \rangle+o(|t-s|)\nonumber \\
&\ \hspace{9cm} (\text{by $f_\tau(Y_t)=\langle \tau, \Y_t \rangle$}).\mlabel{eq:yito4}
\end{align}
Replacing the first term to the right hand side of (\ref{eq:yito3}) with (\ref{eq:yito4}),
\begin{align}
&\ \delta F(Y)_{s,t}\nonumber\\
=&\ \sum_{i=1}^{d}\R{$\Int_s^t$}DF(Y_r):\big(f_{i}(Y_r)\big)dX_{r}^{i}-\sum_{i,j=1}^{d}D\Big(DF(Y_s):\big(f_{i}(Y_s)\big)\Big):\big(f_{j}(Y_s)\big)\langle \X_{s,t}, \tddeuxa{$i$}{$j$}\ \,\rangle\nonumber\\
&\ +\sum_{i,j=1}^{d}DF(Y_s):\Big(Df_i(Y_s):\big(f_j(Y_s)\big)\Big)\langle \X_{s,t}, \tddeuxa{$i$}{$j$}\ \,\rangle+\sum_{i, j=1}^{d}D^2F(Y_s):\big(f_{i}(Y_s), f_{j}(Y_s)\big)\langle \X_{s,t}, \bullet_i\bullet_j\rangle\nonumber\\
&\ +o(|t-s|).\mlabel{eq:yito8}
\end{align}
For the second summand,
\begin{align}
&\ D\Big(DF(Y_s):f_{i}(Y_s)\Big):\big(f_{j}(Y_s)\big) \nonumber\\
=&\ \sum_{{\alpha_1}=1}^n\partial_{\alpha_1}\Big(DF(Y_s):f_{i}(Y_s)\Big)\big(f_{j}(Y_s)\big)^{\alpha_1} \hspace{2cm} (\text{by (\ref{eq:ieq})})\nonumber\\
=&\ \sum_{{\alpha_1}=1}^n\Big(\partial_{\alpha_1}DF(Y_s):f_{i}(Y_s)\Big)\big(f_{j}(Y_s)\big)^{\alpha_1}+\sum_{{\alpha_1}=1}^n\Big(DF(Y_s):\partial_{\alpha_1}f_{i}(Y_s)\Big)\big(f_{j}(Y_s)\big)^{\alpha_1} \nonumber\\
&\ \hspace{10cm} (\text{by the Leibniz rule})\nonumber\\
=&\ \sum_{{\alpha_1}=1}^n\sum_{{\alpha_2}=1}^n\partial_{\alpha_1}\partial_{\alpha_2}F(Y_s)\big(f_{i}(Y_s)\big)^{\alpha_2}\big(f_{j}(Y_s)\big)^{\alpha_1}
+\sum_{{\alpha_2}=1}^n\partial_{\alpha_2}F(Y_s)\Big(\sum_{{\alpha_1}=1}^n\partial_{\alpha_1}f_{i}(Y_s)\big(f_{j}(Y_s)\big)^{\alpha_1}\Big)^{\alpha_2}\nonumber\\
=&\ D^2F(Y_s):\Big(f_{i}(Y_s), f_{j}(Y_s)\Big)+DF(Y_s):\Big(Df_{i}(Y_s):f_{j}(Y_s)\Big) \hspace{2cm} (\text{by (\ref{eq:ieq})}).\mlabel{eq:yito5}
\end{align}
Substituting~(\mref{eq:yito5}) into~(\mref{eq:yito8}) gives
\begin{align}
&\ \delta F(Y)_{s,t}-\sum_{i=1}^{d}\R{$\Int_s^t$}DF(Y_r):f_{i}(Y_r)dX_{r}^{i}\nonumber\\
=&\ -\sum_{i,j=1}^{d}D^2F(Y_s):(f_{i}(Y_s), f_{j}(Y_s))\langle \X_{s,t}, \tddeuxa{$i$}{$j$}\ \,\rangle+\sum_{i, j=1}^{d}D^2F(Y_s):(f_{i}(Y_s), f_{j}(Y_s))\langle \X_{s,t}, \bullet_j\bullet_i\rangle\nonumber\\
=&\ \sum_{i, j=1}^{d}D^2F(Y_s):(f_{i}(Y_s), f_{j}(Y_s))\langle \X_{s,t}, \bullet_j\bullet_i-\tddeuxa{$i$}{$j$}\ \,\rangle+o(|t-s|) \hspace{1cm} (\text{by (\ref{eq:yito5})})\nonumber\\
=&\ \sum_{i, j=1}^{d}D^2F(Y_s):(f_{i}(Y_s), f_{j}(Y_s))\delta \widehat{X}^{(ij)}_{s,t}+o(|t-s|) \hspace{1cm} (\text{by $\delta\widehat{X}^{(ij)}_{s,t}= \langle \X_{s,t}, \bullet_j\bullet_i- \tddeuxa{$i$}{$j$}\ \,  \rangle$}).\mlabel{eq:yito6}
\end{align}
Similar to (\ref{eq:int4}),
\begin{align}
&\ \R{$\Int_s^t$} D^2F(Y_r):(f_{i}(Y_r), f_{j}(Y_r))d\widehat{X}_{r}^{(ij)}\nonumber\\
=&\ \left \langle \etree, D^2F(\Y_s):(f_{i}(\Y_s), f_{j}(\Y_s))  \right \rangle \left \langle \EX_{s,t}, \bullet_{(ij)}  \right \rangle\nonumber\\
&\ +\sum_{k=1}^d\langle \bullet_k,  D^2F(\Y_s):(f_{i}(\Y_s), f_{j}(\Y_s))\rangle \langle \EX_{s,t}, \tddeuxa{$(ij)$}{$k$}\ \ \ \  \rangle+O(|t-s|^{3\alpha})\hspace{0.6cm} (\text{by~(\ref{eq:inte1-})})\nonumber\\
=&\ \left \langle \etree, D^2F(\Y_s):(f_{i}(\Y_s), f_{j}(\Y_s))  \right \rangle \left \langle \EX_{s,t}, \bullet_{(ij)}  \right \rangle+o(|t-s|)\nonumber\\
&\ \hspace{5cm} (\text{by $|\langle \EX_{s,t}, \tddeuxa{$(ij)$}{$k$}\ \ \ \ \rangle|=O(|t-s|^{3\alpha})$ and $3\alpha>1$})\nonumber\\
=&\ D^2F(Y_s):(f_{i}(Y_s), f_{j}(Y_s))\delta \widehat{X}^{(ij)}_{s,t}+o(|t-s|) \hspace{1.5cm} (\text{by Lemma~\ref{lem:ex-x}}).\mlabel{eq:yito7}
\end{align}
Employing (\ref{eq:yito7}) to replace the first summand in the right hand side of (\ref{eq:yito6}) yields
$$\delta F(Y)_{s,t}-\sum_{i=1}^{d}\R{$\Int_s^t$}DF(Y_r):f_{i}(Y_r)dX_{r}^{i}-\sum_{i,j=1}^{d}\R{$\Int_s^t$} D^2F(Y_r):(f_{i}(Y_r), f_{j}(Y_r))d\widehat{X}_{r}^{(ij)}=o(|t-s|),$$
whence with a similar argument to (\ref{eq:ito5}),
\begin{align*}
\delta F(Y)_{s,t}
=&\ \sum_{i=1}^{d}\R{$\Int_s^t$}DF(Y_r):f_{i}(Y_r)dX_{r}^{i}+\sum_{i,j=1}^{d}\R{$\Int_s^t$} D^2F(Y_r):(f_{i}(Y_r), f_{j}(Y_r))d\widehat{X}_{r}^{(ij)}\\
=&\ \R{$\Int_s^t$}DF(Y_r):\Big(\big(f(Y_r) \big)\cdot dX_{r}\Big)+\R{$\Int_s^t$} D^2F(Y_r):\Big(\big(f(Y_r), f(Y_r)\big)\cdot d\widehat{X}_{r}\Big) \hspace{0.5cm} (\text{by (\ref{eq:note2})}).
\end{align*}
 This completes the proof.
\end{proof}

Finally, we arrive at the It\^o formula for general case $F(Y)$ with roughness $\alpha\in (\frac{1}{4},\frac{1}{3}]$. Denote
\begin{equation}
\R{$\Int_s^t$}D^3F(Y_r):\Big(\big(f(Y_r), f(Y_r), f(Y_r)  \big)\cdot d\widetilde{X}_r\Big):=\sum_{i, j, k=1}^{d} \R{$\Int_s^t$} \Big(D^3F(Y_r):\big(f_i(Y_r), f_j(Y_r), f_k(Y_r)\big)\Big)d\widetilde{X}^{(ijk)}_{r},
\label{eq:thirdint}
\end{equation}
where all integrals on the right hand side are Young integrals. Similar to \eqref{eq:+1}, we have
\begin{equation*}
\begin{aligned}
&\ \hmr(\bullet_i \tddeuxa{$j$}{$k$}\ \,+\tddeuxa{$j$}{$k$}\ \, \bullet_i-\tdddeuxa{$i$}{$j$}{$k$}\ \,-\tdtroisuna{$j$}{$k$}{$i$}\ \,-\tdtroisuna{$j$}{$i$}{$k$}\ \,-\tddeuxa{$(ij)$}{$k$}\ \ \ \ -\tddeuxa{$(ji)$}{$k$}\ \ \ \  )\\
=&\ \Big(\bullet_i\otimes  \tddeuxa{$j$}{$k$}\ \,+\bullet_k\otimes \bullet_i\bullet_j+\bullet_i\shuffle\bullet_k\otimes \bullet_j\Big)
+\Big(\bullet_k\otimes \bullet_j\bullet_i+\tddeuxa{$j$}{$k$}\ \,\otimes \bullet_i\Big)\\
&\ -\Big(\bullet_k\otimes  \tddeuxa{$i$}{$j$}\ \,+\tddeuxa{$j$}{$k$}\ \,\otimes\bullet_i\Big)
-\Big(\bullet_i\otimes  \tddeuxa{$j$}{$k$}\ \,+\bullet_i\bullet_k\otimes\bullet_j\Big)
-\Big(\bullet_k\otimes  \tddeuxa{$j$}{$i$}\ \,+\bullet_k\bullet_i\otimes\bullet_j\Big)\\
&\ -\bullet_k\otimes\bullet_{(ij)}-\bullet_k\otimes\bullet_{(ji)}\\
=&\ \bullet_k\otimes\,\Big(\bullet_i\bullet_j+\bullet_j\bullet_i- \tddeuxa{$i$}{$j$}\ \,-\tddeuxa{$j$}{$i$}\ \,-\bullet_{(ij)}-\bullet_{(ji)}\Big).
\end{aligned}
\end{equation*}
Lemma~\ref{lem:ex-x} ensures that
\[
\langle\EX_{s,t}, \bullet_i\bullet_j+\bullet_j\bullet_i- \tddeuxa{$i$}{$j$}\ \,-\tddeuxa{$j$}{$i$}\ \,-\bullet_{(ij)}-\bullet_{(ji)} \rangle=0.
\]
Consequently, Chen's relation and the preceding coproduct identity imply that the increments associated with this element are additive.

\vspace{0.2cm}
For $i,j,k=1,\ldots,d$, define $\cbar{X}^{(ijk)}:[0,T] \to \RR$ by setting 
\begin{equation*}
\cbar{X}^{(ijk)}_{t}:=\langle \EX_{0,t},\bullet_i \tddeuxa{$j$}{$k$}\ \,+\tddeuxa{$j$}{$k$}\ \, \bullet_i-\tdddeuxa{$i$}{$j$}{$k$}\ \,-\tdtroisuna{$j$}{$k$}{$i$}\ \,-\tdtroisuna{$j$}{$i$}{$k$}\ \,-\tddeuxa{$(ij)$}{$k$}\ \ \ \ -\tddeuxa{$(ji)$}{$k$}\ \ \ \  \rangle,\quad \forall t\in [0, T].
\end{equation*}
The additivity established above implies
\begin{equation}
\delta\cbar{X}^{(ijk)}_{s, t}=\langle \EX_{s,t},\bullet_i \tddeuxa{$j$}{$k$}\ \,+\tddeuxa{$j$}{$k$}\ \, \bullet_i-\tdddeuxa{$i$}{$j$}{$k$}\ \,-\tdtroisuna{$j$}{$k$}{$i$}\ \,-\tdtroisuna{$j$}{$i$}{$k$}\ \,-\tddeuxa{$(ij)$}{$k$}\ \ \ \ -\tddeuxa{$(ji)$}{$k$}\ \ \ \  \rangle.
\mlabel{eq:debx}
\end{equation}

\begin{theorem}
Let $\alpha\in (\frac{1}{4},\frac{1}{3}]$, $\X\in \pbrpt $ above $X=(X^1, \ldots, X^d)$ and $\Y \in \cbrpxt$ above $Y$ be a solution to (\ref{eq:BRDE}). Let $\widehat{\X}$ above $\widehat{X}$ in~(\ref{eq:ex-5}) be the bracket extension of\, $\X$ and $F:\RR^n \to \RR^n$ be a smooth function. Then
\begin{equation}
\begin{aligned}
\delta F(Y)_{s,t}
&\ =\R{$\Int_s^t$}DF(Y_r):\Big(\big(f(Y_r) \big)\cdot dX_{r}\Big)+\R{$\Int_s^t$} D^2F(Y_r):\Big(\big(f(Y_r), f(Y_r)\big)\cdot d\widehat{X}_{r}\Big)\\
&\ \ \ \ \ \,+\R{$\Int_s^t$} D^3F(Y_r):\Big(\big(f(Y_r), f(Y_r), f(Y_r)\big)\cdot d\widetilde{X}_{r}\Big)\\
&\ \ \ \ \ \,+\R{$\Int_s^t$} D^2F(Y_r):\Big(\big(f(Y_r), Df(Y_r):f(Y_r) \big)\cdot d\cbar{X}_{r}\Big),
\end{aligned}
\mlabel{eq:integral4}
\end{equation}
where the first and second integrals are from~(\ref{eq:note2}), the third integral is from~(\ref{eq:thirdint}) and the fourth integral is a Young integral.
\mlabel{thm:yitob}
\end{theorem}
\vspace{-0.5cm}

\begin{proof}
First,
\begin{align}
\delta Y_{s,t}\nonumber=&\ \sum_{i=1}^{d}f_{\bullet_i}(Y_s)\langle \X_{s,t}, \bullet_i\rangle+\sum_{i,j=1}^{d}f_{\tddeuxaa{$i$}{$j$}\ \,}(Y_s)\langle \X_{s,t}, \tddeuxa{$i$}{$j$}\ \,\rangle \nonumber \\
&\ +\sum_{i,j,k=1}^{d}f_{\tdddeuxaa{$i$}{$j$}{$k$}\ \,}(Y_s)\langle \X_{s,t}, \tdddeuxa{$i$}{$j$}{$k$}\ \,\rangle+\sum_{i,j,k=1}^{d}f_{\tdtroisuna{$i$}{$k$}{$j$}\ \,}(Y_s)\langle \X_{s,t}, \tdtroisuna{$i$}{$k$}{$j$}\ \,\rangle+o(|t-s|)\nonumber \hspace{1cm}(\text{by (\mref{eq:ftalor})})\nonumber\\
=&\ \sum_{i=1}^{d}f_{i}(Y_s)\langle \X_{s,t}, \bullet_i\rangle+\sum_{i,j=1}^{d}Df_{i}(Y_s):\big(f_{j}(Y_s)\big)\langle \X_{s,t}, \tddeuxa{$i$}{$j$}\ \,\rangle\nonumber\\
&\ +\sum_{i,j,k=1}^{d}Df_{i}(Y_s):\Big(Df_{j}(Y_s):f_{k}(Y_s)\Big)\langle \X_{s,t}, \tdddeuxa{$i$}{$j$}{$k$}\ \,\rangle\nonumber\\
&\ +\sum_{i,j,k=1}^{d}D^2f_{i}(Y_s):\Big(f_{j}(Y_s), f_{k}(Y_s)\Big)\langle \X_{s,t}, \tdtroisuna{$i$}{$k$}{$j$}\ \,\rangle+o(|t-s|) \hspace{0.5cm}(\text{by $f_{\bullet_i} = f_i$ and (\ref{eq:RDE2})}).\mlabel{eq:yitob1}
\end{align}
Employing the third-order Taylor expansion of $F$ at $Y_s$ gives
\begin{align}
\delta F(Y)_{s,t}
=&\ \sum_{i=1}^{n}\partial_iF(Y_s)(Y_{t}^{i}-Y_{s}^{i})+\frac{1}{2}\sum_{i,j=1}^{n}\partial_i\partial_jF(Y_s)(Y_{t}^{i}-Y_{s}^{i})(Y_{t}^{j}-Y_{s}^{j})\nonumber \\
&\ +\frac{1}{6}\sum_{i,j,k=1}^{n}\partial_i\partial_j\partial_kF(Y_s)(Y_{t}^{i}-Y_{s}^{i})(Y_{t}^{j}-Y_{s}^{j})(Y_{t}^{k}-Y_{s}^{k})+O(|Y_t-Y_s|^4)\nonumber\\
=&\ \sum_{i=1}^{n}\partial_iF(Y_s)(Y_{t}^{i}-Y_{s}^{i})+\frac{1}{2}\sum_{i,j=1}^{n}\partial_i\partial_jF(Y_s)(Y_{t}^{i}-Y_{s}^{i})(Y_{t}^{j}-Y_{s}^{j})\nonumber\\
&\ +\frac{1}{6}\sum_{i,j,k=1}^{n}\partial_i\partial_j\partial_kF(Y_s)(Y_{t}^{i}-Y_{s}^{i})(Y_{t}^{j}-Y_{s}^{j})(Y_{t}^{k}-Y_{s}^{k})+o(|t-s|)\nonumber\\
&\ \hspace{6cm} (\text{by $|Y_t-Y_s|=O(|t-s|^\alpha)$ and $4\alpha>1$})\nonumber\\
=&\ DF(Y_s):(\delta Y_{s,t})+\frac{1}{2}D^2F(Y_s):(\delta Y_{s,t}, \delta Y_{s,t})\nonumber\\
&\  +\frac{1}{6}D^3F(Y_s):(\delta Y_{s,t}, \delta Y_{s,t}, \delta Y_{s,t})+o(|t-s|) \quad (\text{by~(\ref{eq:ieq})}).\mlabel{eq:yitob4}
\end{align}
Plugging (\ref{eq:yitob1}) into (\ref{eq:yitob4}),
\begin{align}
\delta F(Y)_{s,t}
=&\ \sum_{i=1}^{d}DF(Y_s):\big(f_{i}(Y_s)\langle \X_{s,t}, \bullet_i\rangle\big)+\sum_{i, j=1}^{d}DF(Y_s):\Big(Df_{i}(Y_s):f_{j}(Y_s)\langle \X_{s,t}, \tddeuxa{$i$}{$j$}\ \,\rangle\Big) \nonumber\\
&\ +\sum_{i,j,k=1}^{d}DF(Y_s):\Big(Df_{i}(Y_s):\big(Df_{j}(Y_s):f_{k}(Y_s)\big)\langle \X_{s,t}, \tdddeuxa{$i$}{$j$}{$k$}\ \,\rangle  \Big)\nonumber\\
&\ +\sum_{i,j,k=1}^{d}DF(Y_s):\Big(D^2f_{i}(Y_s):\big(f_{j}(Y_s), f_{k}(Y_s)\big)\langle \X_{s,t}, \tdtroisuna{$i$}{$k$}{$j$}\ \,\rangle \Big)\nonumber\\
&\ +\frac{1}{2}\sum_{i, j=1}^{d}D^2F(Y_s):(f_{i}(Y_s)\langle \X_{s,t}, \bullet_i\rangle, f_{j}(Y_s)\langle \X_{s,t}, \bullet_j\rangle)  \nonumber\\
&\ +\frac{1}{2}\sum_{i, j, k=1}^{d}D^2F(Y_s):\Big(f_{i}(Y_s)\langle \X_{s,t}, \bullet_i\rangle, Df_{j}(Y_s):f_{k}(Y_s)\langle \X_{s,t}, \tddeuxa{$j$}{$k$}\ \,\rangle\Big) \nonumber\\
&\ +\frac{1}{2}\sum_{i, j, k=1}^{d}D^2F(Y_s):\Big(Df_{j}(Y_s):f_{k}(Y_s)\langle \X_{s,t}, \tddeuxa{$j$}{$k$}\ \,\rangle, f_{i}(Y_s)\langle \X_{s,t}, \bullet_i\rangle\Big) \nonumber\\
&\ +\frac{1}{6}\sum_{i,j,k=1}^{d}D^3F(Y_s):\Big(f_{i}(Y_s)\langle \X_{s,t}, \bullet_i\rangle, f_{j}(Y_s)\langle \X_{s,t}, \bullet_j\rangle, f_{k}(Y_s)\langle \X_{s,t}, \bullet_k\rangle\Big)+o(|t-s|)\nonumber\\
&\  \hspace{6cm} (\text{by Definition~\ref{def:pbrp}~(\ref{it:bitem2}) and $\langle \X_{s,t}, \tau \rangle=0$ for $|\tau|>3$}).\mlabel{eq:yitob2}
\end{align}
The sum of the last four summands is equal to
\begin{align*}
&\ \frac{1}{2}\sum_{i, j=1}^{d}D^2F(Y_s):(f_{i}(Y_s), f_{j}(Y_s))\langle \X_{s,t}, \bullet_i \shuffle \bullet_j\rangle\\
+&\ \frac{1}{2}\sum_{i, j, k=1}^{d}D^2F(Y_s):\Big(f_{i}(Y_s), Df_{j}(Y_s):f_{k}(Y_s)\Big)\langle \X_{s,t}, \bullet_i \shuffle \tddeuxa{$j$}{$k$}\ \,\rangle   \\
+&\ \frac{1}{2}\sum_{i, j, k=1}^{d}D^2F(Y_s):\Big(Df_{j}(Y_s):f_{k}(Y_s), f_{i}(Y_s)\Big)\langle \X_{s,t}, \tddeuxa{$j$}{$k$}\ \,\shuffle \bullet_i\rangle   \\
+&\ \frac{1}{6}\sum_{i, j, k=1}^{d}D^3F(Y_s):\Big(f_{i}(Y_s), f_{j}(Y_s), f_{k}(Y_s)\Big)\langle \X_{s,t}, \bullet_i  \shuffle\bullet_j \shuffle\bullet_k\rangle+o(|t-s|)\\
&\ \hspace{8cm} (\text{by Definition~\ref{def:pbrp}~(\ref{it:bitem2})})\\
=&\ \sum_{i, j=1}^{d}D^2F(Y_s):(f_{i}(Y_s), f_{j}(Y_s))\langle \X_{s,t}, \bullet_j\bullet_i\rangle\\
&\ +\sum_{i, j, k=1}^{d}D^2F(Y_s):\Big(f_{i}(Y_s), Df_{j}(Y_s):f_{k}(Y_s)\Big)\langle \X_{s,t}, \bullet_i \tddeuxa{$j$}{$k$}\ \,\rangle  \\
&\ +\sum_{i, j, k=1}^{d}D^2F(Y_s):\Big(Df_{j}(Y_s):f_{k}(Y_s), f_{i}(Y_s)\Big)\langle \X_{s,t}, \tddeuxa{$j$}{$k$}\ \, \bullet_i\rangle  \\
&\ +\sum_{i,j,k=1}^{d}D^3F(Y_s):\Big(f_{i}(Y_s), f_{j}(Y_s), f_{k}(Y_s)\Big)\langle \X_{s,t}, \bullet_k  \bullet_j \bullet_i\rangle+o(|t-s|).
\end{align*}
Substituting the above equation into (\ref{eq:yitob2}) yields
\begin{align*}
\delta F(Y)_{s,t}=&\
\sum_{i=1}^{d}DF(Y_s):(f_{i}(Y_s))\langle \X_{s,t}, \bullet_i\rangle+\sum_{i, j=1}^{d}DF(Y_s):\Big(Df_{i}(Y_s):f_{j}(Y_s)\Big)\langle \X_{s,t}, \tddeuxa{$i$}{$j$}\ \,\rangle\\
&\ +\sum_{i,j,k=1}^{d}DF(Y_s):\Big(Df_{i}(Y_s):\big(Df_{j}(Y_s):f_{k}(Y_s)\big)  \Big)\langle \X_{s,t}, \tdddeuxa{$i$}{$j$}{$k$}\ \,\rangle\\
&\ +\sum_{i,j,k=1}^{d}DF(Y_s):\Big(D^2f_{i}(Y_s):\big(f_{j}(Y_s), f_{k}(Y_s)\big) \Big)\langle \X_{s,t}, \tdtroisuna{$i$}{$k$}{$j$}\ \,\rangle\\
&\ +\sum_{i, j=1}^{d}D^2F(Y_s):(f_{i}(Y_s), f_{j}(Y_s))\langle \X_{s,t}, \bullet_j\bullet_i\rangle\\
&\ +\sum_{i, j, k=1}^{d}D^2F(Y_s):\Big(f_{i}(Y_s), Df_{j}(Y_s):f_{k}(Y_s)\Big)\langle \X_{s,t}, \bullet_i \tddeuxa{$j$}{$k$}\ \,\rangle  \\
&\ +\sum_{i, j, k=1}^{d}D^2F(Y_s):\Big(Df_{j}(Y_s):f_{k}(Y_s), f_{i}(Y_s)\Big)\langle \X_{s,t}, \tddeuxa{$j$}{$k$}\ \, \bullet_i\rangle   \\
&\ +\sum_{i,j,k=1}^{d}D^3F(Y_s):\Big(f_{i}(Y_s), f_{j}(Y_s), f_{k}(Y_s)\Big)\langle \X_{s,t}, \bullet_k  \bullet_j \bullet_i\rangle+o(|t-s|).
\end{align*}
Regarding the first summand, by Lemma~\ref{lem:compose}, we can define an $\X$-controlled planarly branched rough path
$DF(\Y_s):\big(f_{i}(\Y_s)\big)\in \cbrpxs.$
Then we can define the rough integral $\R{$\Int_s^t$}DF(Y_r):f_{i}(Y_r)dX_{r}^{i}$ by Lemma~\ref{thm:inte1} and have
\begin{align}
&\ \R{$\Int_s^t$}DF(Y_r):f_{i}(Y_r)dX_{r}^{i}\nonumber\\
=&\  \langle \etree, DF(\Y_s):f_{i}(\Y_s)\rangle \langle \X_{s,t}, \bullet_i \rangle+\sum_{j=1}^{d} \langle \bullet_j, DF(\Y_s):f_{i}(\Y_s) \rangle  \langle \X_{s,t}, \tddeuxa{$i$}{$j$}\ \,  \rangle\nonumber\\
&\ +\sum_{j,k=1}^{d} \langle \bullet_j\bullet_k, DF(\Y_s):f_{i}(\Y_s) \rangle \langle \X_{s,t},\tdtroisuna{$i$}{$k$}{$j$}\ \,\rangle+\sum_{j,k=1}^{d} \langle \tddeuxa{$j$}{$k$}\ \,, DF(\Y_s):f_{i}(\Y_s) \rangle \langle \X_{s,t},\tdddeuxa{$i$}{$j$}{$k$}\ \,\rangle+O(|t-s|^{4\alpha})\nonumber\\
&\ \hspace{10cm} (\text{by (\ref{eq:inte1-})})\nonumber\\
=&\  \langle \etree, DF(\Y_s):f_{i}(\Y_s)\rangle \langle \X_{s,t}, \bullet_i \rangle+\sum_{j=1}^{d} \langle \bullet_j, DF(\Y_s):f_{i}(\Y_s) \rangle  \langle \X_{s,t}, \tddeuxa{$i$}{$j$}\ \,  \rangle\nonumber\\
&\ +\sum_{j,k=1}^{d} \langle \bullet_j\bullet_k, DF(\Y_s):f_{i}(\Y_s) \rangle \langle \X_{s,t},\tdtroisuna{$i$}{$k$}{$j$}\ \,\rangle+\sum_{j,k=1}^{d} \langle \tddeuxa{$j$}{$k$}\ \,, DF(\Y_s):f_{i}(\Y_s) \rangle \langle \X_{s,t},\tdddeuxa{$i$}{$j$}{$k$}\ \,\rangle+o(|t-s|)\nonumber\\
&\ \hspace{10cm} (\text{by $4\alpha>1$})\nonumber\\
=&\ DF(Y_s):f_{i}(Y_s) \langle \X_{s,t}, \bullet_i \rangle+\sum_{j=1}^{d}D(DF(Y_s):f_{i}(Y_s)):\langle \bullet_j,\Y_s \rangle \langle \X_{s,t}, \tddeuxa{$i$}{$j$}\ \, \rangle\nonumber\\
&\ +\sum_{j,k=1}^{d}\Big(D(DF(Y_s):f_{i}(Y_s)):\langle \bullet_j\bullet_k, \Y_s \rangle+D^2(DF(Y_s):f_{i}(Y_s)):(\langle \bullet_j,\Y_s \rangle\langle \bullet_k,\Y_s \rangle)   \Big)\langle \X_{s,t},\tdtroisuna{$i$}{$k$}{$j$}\ \,\rangle\nonumber\\
&\ +\sum_{j,k=1}^{d}D(DF(Y_s):f_{i}(Y_s)):\langle \tddeuxa{$j$}{$k$}\ \,, \Y_s \rangle\langle \X_{s,t},\tdddeuxa{$i$}{$j$}{$k$}\ \,\rangle+o(|t-s|)\hspace{1cm} (\text{by (\ref{eq:zcbrp})})\nonumber\\
=&\ DF(Y_s):f_{i}(Y_s) \langle \X_{s,t}, \bullet_i \rangle+\sum_{j=1}^{d}D(DF(Y_s):f_{i}(Y_s)):f_{j}(Y_s)\langle \X_{s,t}, \tddeuxa{$i$}{$j$}\ \,   \rangle\nonumber\\
&\ +\sum_{j,k=1}^{d}D^2(DF(Y_s):f_{i}(Y_s)):(f_{j}(Y_s),f_{k}(Y_s)) \langle \X_{s,t},\tdtroisuna{$i$}{$k$}{$j$}\ \,\rangle\nonumber\\
&\ +\sum_{j,k=1}^{d}D(DF(Y_s):f_{i}(Y_s)):f_{\tddeuxa{$j$}{$k$}\ \,}(Y_s)\langle \X_{s,t},\tdddeuxa{$i$}{$j$}{$k$}\ \,\rangle+o(|t-s|). \mlabel{eq:yitob3}
\end{align}
Here in the last step, the first part of the third summand vanishes by~(\ref{eq:RDE1}) and others employ~(\mref{eq:ftau}).
Using (\ref{eq:yitob3}) in place of the first term to the right hand side of (\ref{eq:yitob2}),
\begin{align}
\delta F(Y)_{s,t}
=&\ \sum_{i=1}^{d}\R{$\Int_s^t$}DF(Y_r):f_{i}(Y_r)dX_{r}^{i}-\sum_{i,j=1}^{d}D(DF(Y_s):f_{i}(Y_s)):f_{j}(Y_s)\langle \X_{s,t}, \tddeuxa{$i$}{$j$}\ \,\rangle \nonumber\\
&-\sum_{i,j,k=1}^{d}D^2(DF(Y_s):f_{i}(Y_s)):(f_{j}(Y_s),f_{k}(Y_s)) \langle \X_{s,t},\tdtroisuna{$i$}{$k$}{$j$}\ \,\rangle\nonumber\\
&-\sum_{i,j,k=1}^{d}D(DF(Y_s):f_{i}(Y_s)):(Df_j(Y_s):f_k(Y_s))\langle \X_{s,t},\tdddeuxa{$i$}{$j$}{$k$}\ \,\rangle \nonumber\\
&\ +\sum_{i, j=1}^{d}DF(Y_s):\Big(Df_{i}(Y_s):f_{j}(Y_s)\Big)\langle \X_{s,t}, \tddeuxa{$i$}{$j$}\ \,\rangle\nonumber\\
&\ +\sum_{i,j,k=1}^{d}DF(Y_s):\Big(Df_{i}(Y_s):\big(Df_{j}(Y_s):f_{k}(Y_s)\big)  \Big)\langle \X_{s,t}, \tdddeuxa{$i$}{$j$}{$k$}\ \,\rangle\nonumber\\
&\ +\sum_{i,j,k=1}^{d}DF(Y_s):\Big(D^2f_{i}(Y_s):\big(f_{j}(Y_s), f_{k}(Y_s)\big) \Big)\langle \X_{s,t}, \tdtroisuna{$i$}{$k$}{$j$}\ \,\rangle\nonumber\\
&\ +\sum_{i, j=1}^{d}D^2F(Y_s):(f_{i}(Y_s), f_{j}(Y_s))\langle \X_{s,t}, \bullet_j\bullet_i\rangle\nonumber\\
&\ +\sum_{i, j, k=1}^{d}D^2F(Y_s):\Big(f_{i}(Y_s), Df_{j}(Y_s):f_{k}(Y_s)\Big)\langle \X_{s,t}, \bullet_i \tddeuxa{$j$}{$k$}\ \,\rangle   \nonumber\\
&\ +\sum_{i, j, k=1}^{d}D^2F(Y_s):\Big(Df_{j}(Y_s):f_{k}(Y_s), f_{i}(Y_s)\Big)\langle \X_{s,t}, \tddeuxa{$j$}{$k$}\ \, \bullet_i\rangle   \nonumber\\
&\ +\sum_{i,j,k=1}^{d}D^3F(Y_s):\Big(f_{i}(Y_s), f_{j}(Y_s), f_{k}(Y_s)\Big)\langle \X_{s,t}, \bullet_k  \bullet_j \bullet_i\rangle+o(|t-s|).\mlabel{eq:yitob5}
\end{align}
For the third and fourth summands, employing the same argument to (\ref{eq:yito5}) yields
\begin{align}
&\ D^2\big(DF(Y_s):f_{i}(Y_s)\big):\big(f_{j}(Y_s),f_{k}(Y_s)\big)\nonumber \\
=&\ D^3F(Y_s):\Big(f_{i}(Y_s),f_{j}(Y_s),f_{k}(Y_s)
\Big)+DF(Y_s):\Big(D^2f_{i}(Y_s):\big(f_{j}(Y_s),f_{k}(Y_s)\big)\Big)\nonumber \\
&\ +D^2F(Y_s):\Big(Df_{i}(Y_s):\big(f_{j}(Y_s)\big), f_{k}(Y_s)\Big)+D^2F(Y_s):\Big(f_{j}(Y_s), Df_{i}(Y_s):\big(f_{k}(Y_s)\big)\Big)
\mlabel{eq:yitoo5}
\end{align}
and
\begin{align}
&\ D\big(DF(Y_s):f_{i}(Y_s)\big):\big(Df_j(Y_s):f_k(Y_s)\big)\nonumber\\
 =&\ D^2F(Y_s):\Big(f_{i}(Y_s), Df_j(Y_s):\big(f_k(Y_s)\big)\Big)+DF(Y_s):\Big(Df_{i}(Y_s):\big(Df_j(Y_s):f_k(Y_s)\big) \Big).\mlabel{eq:yitooo5}
\end{align}
Inserting (\ref{eq:yito5}), (\ref{eq:yitoo5}) and (\ref{eq:yitooo5}) into (\ref{eq:yitob5}),
\begin{align}
\delta F(Y)_{s,t}
=&\ \sum_{i=1}^{d}\R{$\Int_s^t$}DF(Y_r):f_{i}(Y_r)dX_{r}^{i}+\sum_{i, j=1}^{d}D^2F(Y_s):\Big(f_{i}(Y_s), f_{j}(Y_s)\Big)\langle \X_{s,t}, \bullet_j\bullet_i-\tddeuxa{$i$}{$j$}\ \,\rangle \nonumber\\
&\ +\sum_{i, j, k=1}^{d}D^2F(Y_s):\Big(f_{i}(Y_s), \big(Df_{j}(Y_s):f_{k}(Y_s)\big)\Big)\langle \X_{s,t}, \bullet_i \tddeuxa{$j$}{$k$}\ \,+\tddeuxa{$j$}{$k$}\ \, \bullet_i-\tdddeuxa{$i$}{$j$}{$k$}\ \,-\tdtroisuna{$j$}{$k$}{$i$}\ \,-\tdtroisuna{$j$}{$i$}{$k$}\ \,\rangle   \nonumber\\
&\ +\sum_{i,j,k=1}^{d}D^3F(Y_s):\Big(f_{i}(Y_s), f_{j}(Y_s), f_{k}(Y_s)\Big)\langle \X_{s,t}, \bullet_k  \bullet_j \bullet_i-\tdtroisuna{$i$}{$k$}{$j$}\ \,\rangle+o(|t-s|).\mlabel{eq:yitob6}
\end{align}
Similar to (\ref{eq:intb4}),
\begin{align}
&\ \R{$\Int_s^t$} D^2F(Y_r):\big(f_{i}(Y_r), f_{j}(Y_r)\big)d\widehat{X}_{r}^{(ij)}\nonumber\\
=&\  \langle \etree, D^2F(\Y_s):\big(f_{i}(\Y_s), f_{j}(\Y_s)\big) \rangle \langle \EX_{s,t}, \bullet_{(ij)} \rangle+\sum_{k=1}^{d} \langle \bullet_k, D^2F(\Y_s):\big(f_{i}(\Y_s), f_{j}(\Y_s)\big) \rangle \langle \EX_{s,t}, \tddeuxa{$(ij)$}{$k$}\ \ \ \ \ \rangle\nonumber\\
&\ +\sum_{k,m=1}^{d}\langle \bullet_k\bullet_m, D^2F(\Y_s):\big(f_{i}(\Y_s), f_{j}(\Y_s)\big)   \rangle \langle \EX_{s,t}, \tdtroisuna{$(ij)$}{$m$}{$k$}\ \ \ \rangle\nonumber\\
&\ +\sum_{k,m=1}^{d}\langle \tddeuxa{$k$}{$m$}\ \,, D^2F(\Y_s):\big(f_{i}(\Y_s), f_{j}(\Y_s)\big)   \rangle \langle \EX_{s,t}, \tdddeuxa{$(ij)$}{$k$}{$m$}\ \ \ \  \rangle+O(|t-s|^{4\alpha})\nonumber\\
=&\  \langle \etree, D^2F(\Y_s):\big(f_{i}(\Y_s), f_{j}(\Y_s)\big) \rangle \langle \EX_{s,t}, \bullet_{(ij)} \rangle\nonumber\\
&\ +\sum_{k=1}^{d} \langle \bullet_k, D^2F(\Y_s):\big(f_{i}(\Y_s), f_{j}(\Y_s)\big) \rangle \langle \EX_{s,t}, \tddeuxa{$(ij)$}{$k$}\ \ \ \ \ \rangle+o(|t-s|)\nonumber\\
&\ \hspace{3cm} (\text{by $|\langle \EX_{s,t}, \tdtroisuna{$(ij)$}{$m$}{$k$}\ \ \ \rangle|=O(|t-s|^{4\alpha})$, $|\langle \EX_{s,t}, \tdddeuxa{$(ij)$}{$k$}{$m$}\ \ \ \  \rangle|=O(|t-s|^{4\alpha})$ and $4\alpha>1$})\nonumber\\
=&\ D^2F(Y_s):(f_{i}(Y_s), f_{j}(Y_s))\langle \X_{s,t},\bullet_j\bullet_i- \tddeuxa{$i$}{$j$}\ \, \rangle\nonumber\\
&\ +\sum_{k=1}^{d}D\Big(D^2F(Y_s):(f_{i}(Y_s), f_{j}(Y_s))\Big):f_{k}(Y_s) \langle \EX_{s,t}, \tddeuxa{$(ij)$}{$k$}\ \ \ \  \rangle+o(|t-s|). \hspace{1cm} (\text{by (\ref{eq:zcbrp})})\nonumber\\
=&\ D^2F(Y_s):(f_{i}(Y_s), f_{j}(Y_s))\langle \X_{s,t},\bullet_j\bullet_i- \tddeuxa{$i$}{$j$}\ \, \rangle+\sum_{k=1}^{d}D^3F(Y_s):\Big(f_{i}(Y_s), f_{j}(Y_s), f_{k}(Y_s)\Big)\langle \EX_{s,t}, \tddeuxa{$(ij)$}{$k$}\ \ \ \  \rangle\nonumber\\
&\ +\sum_{k=1}^{d}D^2F(Y_s):\Big(\big(Df_{i}(Y_s):f_{k}(Y_s)\big), f_{j}(Y_s)\Big) \langle \EX_{s,t}, \tddeuxa{$(ij)$}{$k$}\ \ \ \  \rangle\nonumber\\
&\ +\sum_{k=1}^{d}D^2F(Y_s):\Big(\big(Df_{j}(Y_s):f_{k}(Y_s)\big), f_{i}(Y_s)\Big) \langle \EX_{s,t}, \tddeuxa{$(ij)$}{$k$}\ \ \ \  \rangle
+o(|t-s|) \nonumber  \\
&\ \hspace{7cm} (\text{by the Leibniz rule, as in (3.31)}).
\mlabel{eq:itob5}
\end{align}
Using (\ref{eq:itob5}) to substitute $D^2F(Y_s):(f_{i}(Y_s), f_{j}(Y_s))\langle \X_{s,t},\bullet_j\bullet_i- \tddeuxa{$i$}{$j$}\ \, \rangle$ in~(\ref{eq:yitob6}),
\begin{align}
&\ \delta F(Y)_{s,t}-\sum_{i=1}^{d}\R{$\Int_s^t$}DF(Y_r):f_{i}(Y_r)dX_{r}^{i}-\sum_{i, j=1}^{d}\R{$\Int_s^t$} D^2F(Y_r):\Big(f_{i}(Y_r), f_{j}(Y_r)\Big)d\widehat{X}_{r}^{(ij)}  \nonumber\\
=&\ \sum_{i, j, k=1}^{d}D^2F(Y_s):\Big(f_{i}(Y_s), Df_{j}(Y_s):f_{k}(Y_s)\Big)\langle \EX_{s,t}, \bullet_i \tddeuxa{$j$}{$k$}\ \,+\tddeuxa{$j$}{$k$}\ \, \bullet_i-\tdddeuxa{$i$}{$j$}{$k$}\ \,-\tdtroisuna{$j$}{$k$}{$i$}\ \,-\tdtroisuna{$j$}{$i$}{$k$}\ \,-\tddeuxa{$(ij)$}{$k$}\ \ \ \ -\tddeuxa{$(ji)$}{$k$}\ \ \ \  \rangle   \nonumber\\
&\ +\sum_{i,j,k=1}^{d}D^3F(Y_s):\Big(f_{i}(Y_s), f_{j}(Y_s), f_{k}(Y_s)\Big)\langle \EX_{s,t}, \bullet_k  \bullet_j \bullet_i-\tdtroisuna{$i$}{$k$}{$j$}\ \,-\tddeuxa{$(ij)$}{$k$}\ \ \ \ \rangle+o(|t-s|)\nonumber\\
=&\ \sum_{i, j, k=1}^{d}D^2F(Y_s):\Big(f_{i}(Y_s), Df_{j}(Y_s):f_{k}(Y_s)\Big)\delta\cbar{X}^{(ijk)}_{s,t} \nonumber\\
&\ +\sum_{i,j,k=1}^{d}D^3F(Y_s):\Big(f_{i}(Y_s), f_{j}(Y_s), f_{k}(Y_s)\Big)\delta \widetilde{X}^{(ijk)}_{s,t}+o(|t-s|)
\hspace{0.2cm} (\text{by~(\ref{eq:xhat}) and~(\ref{eq:debx})}).\mlabel{eq:itob4}
\end{align}
Since $\alpha+3\alpha>1$, similar to (\ref{eq:intb16}), we can define two Young integrals~\cite{You} of the two terms of the right hand side of~(\ref{eq:itob4}) as:
\begin{align}
&\ \R{$\Int_s^t$} D^2F(Y_r):\Big(f_{i}(Y_r), Df_{j}(Y_r):f_{k}(Y_r)\Big)d\cbar{X}^{(ijk)}_{r}\nonumber\\
:=&\ \lim_{|\pi | \to 0} \sum_{[t_u, t_{u+1}]\in \pi}D^2F(Y_{t_u}):\Big(f_{i}(Y_{t_u}), Df_{j}(Y_{t_u}):f_{k}(Y_{t_u})\Big)\delta\cbar{X}^{(ijk)}_{t_u,t_{u+1}}\nonumber\\
=&\ D^2F(Y_s):\Big(f_{i}(Y_s), Df_{j}(Y_s):f_{k}(Y_s)\Big)\delta\cbar{X}^{(ijk)}_{s,t}+o(|t-s|) \mlabel{eq:itob7}
\end{align}
and
\begin{align}
&\ \R{$\Int_s^t$} D^3F(Y_r):\Big(f_{i}(Y_r), f_{j}(Y_r), f_{k}(Y_r)\Big)d\widetilde{X}^{(ijk)}_{r} \nonumber\\
:=&\ \lim_{|\pi | \to 0} \sum_{[t_u, t_{u+1}]\in \pi}D^3F(Y_{t_u}):\Big(f_{i}(Y_{t_u}), f_{j}(Y_{t_u}), f_{k}(Y_{t_u})\Big)\delta\widetilde{X}^{(ijk)}_{t_u,t_{u+1}}\nonumber\\
=&\ D^3F(Y_s):\Big(f_{i}(Y_s), f_{j}(Y_s), f_{k}(Y_s)\Big)\delta\widetilde{X}^{(ijk)}_{s,t}+o(|t-s|),\mlabel{eq:itob8}
\end{align}
where $\pi$ is an arbitrary partition of $[0, T]$. Plugging (\ref{eq:itob7}) and (\ref{eq:itob8}) into the last equation of (\ref{eq:itob4}),
\begin{align*}
&\ \delta F(Y)_{s,t}-\sum_{i=1}^{d}\R{$\Int_s^t$}DF(Y_r):f_{i}(Y_r)dX_{r}^{i}-\sum_{i, j=1}^{d}\R{$\Int_s^t$} D^2F(Y_r):\Big(f_{i}(Y_r), f_{j}(Y_r)\Big)d\widehat{X}_{r}^{(ij)}\\
=&\ \sum_{i,j,k=1}^{d} \R{$\Int_s^t$} D^2F(Y_r):\Big(f_{i}(Y_r), Df_{j}(Y_r):f_{k}(Y_r)\Big)d\cbar{X}^{(ijk)}_{r}\\
&\ + \sum_{i,j,k=1}^{d}\R{$\Int_s^t$} D^3F(Y_r):\Big(f_{i}(Y_r), f_{j}(Y_r), f_{k}(Y_r)\Big)d\widetilde{X}^{(ijk)}_{r}+o(|t-s|).
\end{align*}
With a similar argument to deal with~(\ref{eq:ito5}),
\begin{align*}
&\ \delta F(Y)_{s,t}\\
=&\ \sum_{i=1}^{d}\R{$\Int_s^t$}DF(Y_r):f_{i}(Y_r)dX_{r}^{i}+\sum_{i, j=1}^{d}\R{$\Int_s^t$} D^2F(Y_r):\Big(f_{i}(Y_r), f_{j}(Y_r)\Big)d\widehat{X}_{r}^{(ij)}\\
&\ +\sum_{i,j,k=1}^{d} \R{$\Int_s^t$} D^2F(Y_r):\Big(f_{i}(Y_r), Df_{j}(Y_r):f_{k}(Y_r)\Big)d\cbar{X}^{(ijk)}_{r}\\
&\ + \sum_{i,j,k=1}^{d}\R{$\Int_s^t$} D^3F(Y_r):\Big(f_{i}(Y_r), f_{j}(Y_r), f_{k}(Y_r)\Big)d\widetilde{X}^{(ijk)}_{r} \\
=&\ \R{$\Int_s^t$}DF(Y_r):\Big(\big(f(Y_r) \big)\cdot dX_{r}\Big)+\R{$\Int_s^t$} D^2F(Y_r):\Big(\big(f(Y_r), f(Y_r)\big)\cdot d\widehat{X}_{r}\Big)\\
&\ +\R{$\Int_s^t$} D^3F(Y_r):\Big(\big(f(Y_r), f(Y_r), f(Y_r)\big)\cdot d\widetilde{X}_{r}\Big)+\R{$\Int_s^t$} D^2F(Y_r):\Big(\big(f(Y_r), Df(Y_r):f(Y_r) \big)\cdot d\cbar{X}_{r}\Big)\\
&\ \hspace{10cm} (\text{by (\ref{eq:note2})}).
\end{align*}
This completes the proof.
\end{proof}

\vskip 0.2in

\noindent
{\bf Acknowledgments.}  This work is supported by the National Natural Science Foundation of China (12571019), the Natural Science Foundation of Gansu Province (25JRRA644) and Innovative Fundamental Research Group Project of Gansu Province (23JRRA684). We are grateful to the anonymous reviewers for their insightful suggestions.

\noindent
{\bf Declaration of interests. } The authors have no conflicts of interest to disclose.

\noindent
{\bf Data availability. } Data sharing is not applicable as no new data were created or analyzed.

\end{document}